\RequirePackage[l2tabu, orthodox]{nag}
\documentclass{amsart}

\usepackage[utf8]{inputenc}
\usepackage[T1]{fontenc}
\usepackage{color}
\usepackage{commath}
\usepackage[square, numbers]{natbib}
\usepackage[stretch=10, shrink=10]{microtype}
\usepackage{todonotes}
\usepackage{doi}
\usepackage{cases}
\usepackage{booktabs}
\usepackage{enumitem}
\usepackage{mathtools}
\usepackage{tikz}
\usepackage{graphicx}
\usepackage[labelformat=simple]{subcaption}
\usepackage{multirow}
\usepackage{floatrow}
\usepackage[margin=3cm]{geometry}

\DeclareUrlCommand\doi{\def\UrlLeft##1\UrlRight{doi:\href{http://dx.doi.org/##1}{##1}}\urlstyle{rm}}
\newtheorem{theo}{Theorem}[section]
\newtheorem{prop}[theo]{Proposition}
\newtheorem{lemm}[theo]{Lemma}
\newtheorem{coro}[theo]{Corollary}
\newtheorem{fact}[theo]{Fact}

\theoremstyle{definition}

\theoremstyle{remark}
\newtheorem{rema}[theo]{Remark}
\newtheorem{example}[theo]{Example}

\numberwithin{equation}{section}

\renewcommand{\Re}{\operatorname{Re}}
\renewcommand{\Im}{\operatorname{Im}}
\newcommand{\phoro}[1]{\mathtt{#1}}
\newcommand{\psinh}{\operatorname{\phoro{sinh}}}
\newcommand{\pcosh}{\operatorname{\phoro{cosh}}}
\newcommand{\ptanh}{\operatorname{\phoro{tanh}}}
\newcommand{\psin}{\operatorname{\phoro{sin}}}
\newcommand{\pcos}{\operatorname{\phoro{cos}}}
\newcommand{\ptan}{\operatorname{\phoro{tan}}}
\newcommand{\resp}[1]{\textup{(}resp. #1\textup{)}}
\newcommand{\circlenum}[1]{\raisebox{.5pt}{\textcircled{\raisebox{-.9pt} {#1}}}}
\newcommand*\circled[1]{\tikz[baseline=(char.base)]{
    \node[shape=circle,draw,inner sep=1.5pt,label=center:#1] (char) {\phantom{0}};}}
\newcommand{\inner}[2]{\left\langle{#1},{#2}\right\rangle}
\newcommand{\vect}[1]{\boldsymbol{#1}}

\newsavebox\mybox

\setlist[itemize]{leftmargin=2.2em}
  
\title[Analysis of timelike Thomsen surfaces]{Analysis of timelike Thomsen surfaces \\-- with deformations and singularities}
\author{Shintaro Akamine}
\address[Shintaro Akamine]{Graduate School of Mathematics, Nagoya University, Chikusa-ku, Nagoya 464-8502, Japan}
\email{s-akamine@math.nagoya-u.ac.jp}
 
\author{Joseph Cho}
\address[Joseph Cho]{Department of Mathematics, Graduate School of Science, Kobe University, 1-1 Rokkodai-cho, Nada-ku, Kobe 657-8501, Japan}
\email{joseph.cho@stu.kobe-u.ac.jp}
 
\author{Yuta Ogata}
\address[Yuta Ogata]{Department of Science and Technology, National Institute of Technology, Okinawa College, 905 Henoko, Nago, Okinawa 905-2171, Japan}
\email{y.ogata@okinawa-ct.ac.jp}

\begin{document}


\begin{abstract}
Timelike Thomsen surfaces are timelike minimal surfaces that are also affine minimal.
In this paper, we make use of both the Lorentz conformal coordinates and the null coordinates, and their respective representation theorems of timelike minimal surfaces, to obtain a complete global classification of these surfaces and to characterize them using a geometric invariant called lightlike curvatures.
As a result, we reveal the relationship between timelike Thomsen surfaces, and timelike minimal surfaces with planar curvature lines.
As an application, we give a deformation of null curves preserving the pseudo-arclength parametrization and the constancy of the lightlike curvatures.
\end{abstract}

\maketitle

\section{Introduction}
The relationship between two particular classes of minimal surfaces in Euclidean $3$-space $\mathbb{R}^3$, those with planar curvature lines and those that are also affine minimal, has been known since the early twentieth century.
Minimal surfaces with planar curvature lines were first studied by Bonnet, who noticed that the planar curvature lines transform into orthogonal systems of cycles on the sphere under the Gauss map \cite{bonnet_observations_1855}.
Using this fact, Bonnet found the class of minimal surfaces now often referred to as Bonnet minimal surfaces; however, in his work, he left out the well-known Enneper surface \cite{enneper_untersuchungen_1878}, which also has planar curvature lines \cite[\S 175]{nitsche_vorlesungen_1975}.
On the other hand, Thomsen studied minimal surfaces that are also affine minimal \cite{thomsen_uber_1923}, now referred to as Thomsen surfaces.
Thomsen noted that the asymptotic lines of Thomsen surfaces transform into orthogonal systems of cycles on the $2$-sphere under the Gauss map, showing that Thomsen surfaces are conjugate minimal surfaces of those with planar curvature lines \cite[\S 71]{blaschke_vorlesungen_1923}.
Finally, works such as \cite{schaal_ennepersche_1973, barthel_thomsensche_1980} have shown that there exists a deformation consisting exactly of the Thomsen surfaces.

One can also consider the analogous results for maximal surfaces in Minkowski $3$-space $\mathbb{R}^{2,1}$, which are spacelike surfaces with zero mean curvature by definition.
Leite has classified all maximal surfaces with planar curvature lines in \cite{leite_surfaces_2015}.
Then, Manhart showed that maximal Thomsen surfaces, defined as maximal surfaces that are also affine minimal, are conjugate maximal surfaces of those with planar curvature lines \cite{manhart_bonnet-thomsen_2015}.
And finally in \cite{cho_deformation_2018}, it was shown that there is a deformation consisting exactly of the maximal surfaces with planar curvature lines.

In this paper, we consider the timelike minimal analogue of the two classes of surfaces in Minkowski $3$-space, and clarify their relationship.
We first focus on the class of timelike minimal surfaces with planar curvature lines, and consider its classification.
To achieve this, we use the following method: First, as in \cite{cho_deformation_2017, cho_deformation_2018} (see also \cite{abresch_constant_1987, wente_constant_1992}), using the Lorentz conformal coordinates, we express the timelike minimality condition and the planar curvature line condition via a system of partial differential equations in terms of the Lorentz conformal factor.
Then as in \cite{cho_deformation_2018} (see also \cite{wente_counterexample_1986, walter_explicit_1987}), from the solutions of the system of partial differential equations, we show and utilize the existence of axial directions to recover the Weierstrass data \cite{weierstrass_untersuchungen_1866} for the Weierstrass-type representation for timelike minimal surfaces given by Konderak \cite{konderak_weierstrass_2005} (see Fact \ref{fact:Weierstrass1}).
With the Weierstrass data, we give a complete classification of all timelike minimal surfaces with planar curvature lines (see Theorem \ref{theo:wData}).

Then we switch our attention to the class of timelike Thomsen surfaces, defined by Magid in \cite{magid_timelike_1991-1} as timelike minimal surfaces that are also affine minimal.
In his work, Magid considered the null coordinates representation of timelike minimal surfaces found by McNertney in \cite{mcnertney_one-parameter_1980} (see Fact \ref{fact:Weierstrass2}), where a timelike minimal surface is obtained via two generating null curves.
Using this representation, he applied the result given by Manhart in \cite{manhart_affinminimalruckungsflachen_1985} on affine minimal surfaces of particular form, and obtained an explicit parametrization for the generating null curves of timelike Thomsen surfaces.


Therefore, to investigate the relationship between the two classes of timelike minimal surfaces, we now shift the focus to null coordinates.
We first characterize timelike minimal surfaces with planar curvature lines in terms of geometric invariants of their generating null curves, called \emph{lightlike curvatures} (see Theorem \ref{thm:curvature_characterization2}).
As an application, we obtain deformations of null curves preserving the pseudo-arclength parametrization and the constancy of lightlike curvatures.
Then, interpreting Magid's result on timelike Thomsen surfaces in terms of lightlike curvatures, we reveal a relationship between the two classes of timelike minimal surfaces that differs from that of the minimal case in $\mathbb{R}^3$ and the maximal case in $\mathbb{R}^{2,1}$ (see Theorem \ref{thm:Thomsen_Bonnet_relation}).

In the appendix, similar to \cite{cho_deformation_2018}, we use the axial directions to show that there exists a deformation consisting exactly of all timelike Thomsen surfaces (see Theorem \ref{theo:deformation} and Corollary \ref{cor:deformationThomsen}).
On the other hand, it is possible to consider the singularities appearing on timelike minimal surfaces by viewing the surfaces as \emph{generalized timelike minimal surfaces} as defined in \cite[Definition 2.4]{kim_spacelike_2011}.
Furthermore, in \cite{takahashi_tokuitenwo_2012}, \emph{minfaces} were defined as a class of timelike minimal surfaces admitting certain types of singularities, a timelike minimal analogue of \emph{maxfaces} defined by Umehara and Yamada in \cite[Definition 2.2]{umehara_maximal_2006} for maximal surfaces.
It is known that every minface is a generalized timelike minimal surface; however, there exist generalized timelike minimal surfaces that are not minfaces on their domains (see, for example, \cite[Example 2.7]{kim_spacelike_2011}).
By showing that timelike Thomsen surfaces are minfaces, we recognize the types of singularities appearing on these surfaces, using the criterion introduced in \cite{takahashi_tokuitenwo_2012} (see Theorem \ref{theo:singularityType} and Corollary \ref{cor:singularities_Bonnet}).

\section{Timelike minimal surfaces with planar curvature lines}
In this section, we aim to completely classify timelike minimal surfaces with planar curvature lines.
To achieve this, we propose the following method:
First, we derive a system of partial differential equations for the Lorentz conformal factor from the integrability condition for timelike minimal surfaces and the planar curvature line condition.
Then, using the explicit solutions of the Lorentz conformal factor, we calculate the unit normal vector, and then recover the Weierstrass data using the notion of axial directions.
In doing so, we show the existence of axial directions for these surfaces; by normalizing these axial directions, we eliminate the freedom of isometry in the ambient space, and complete the classification.
The techniques used in this section mirror those of \cite{cho_deformation_2017, cho_deformation_2018}; therefore, we do not explicitly state all the proofs; in place, we sometimes state the outlines of proofs.

\subsection{Paracomplex analysis}
First, we briefly introduce the set of paracomplex numbers $\mathbb{C}'$, and the theory of paracomplex analysis.
For a more detailed introduction, we refer the readers to works such as \cite{akamine_behavior_nodate, inoguchi_timelike_2004, konderak_weierstrass_2005, yasumoto_weierstrass-type_2018}.

We consider the set of paracomplex numbers $\mathbb{C}'$
	\[
		\mathbb{C}' := \{ z = x + j y : x, y \in \mathbb{R} \}
	\]
where $j$ is the imaginary unit such that $j^2 = 1$.
Let $z = x + jy$ denote any paracomplex number.
We call $\Re z := x$ and $\Im z := y$ the \emph{real} and \emph{imaginary parts} of $z$, respectively; furthermore, analogous to the set of complex numbers, we use $\bar{z} := x - jy$ to denote the \emph{paracomplex conjugate} of $z$.
In this paper, we denote the \emph{squared norm} of $z$ as $| z |^2 = z \bar{z} = x^2 - y^2$, which may not necessarily be positive.

We also have the paracomplex Wirtinger derivatives $\partial_z := \frac{1}{2}\left(\partial_x + j \partial_y \right)$ and $\partial_{\bar{z}} := \frac{1}{2}\left(\partial_x - j \partial_y \right)$.
Given a paracomplex function (typeset using typewriter font throughout the paper) $\phoro{f} : \Sigma \subset \mathbb{C}' \to \mathbb{C}'$ where $\Sigma$ is a simply-connected domain, we call $\phoro{f}$ \emph{paraholomorphic} if $\phoro{f}$ satisfies the Cauchy-Riemann type conditions,
	\begin{equation}\label{eqn:cauchyRiemann}
		\phoro{f}_{\bar{z}} = \partial_{\bar{z}} \phoro{f} = 0.
	\end{equation}
Furthermore, following \cite{takahashi_tokuitenwo_2012}, we call a function $\phoro{f} : \Sigma \to \mathbb{C}'$ \emph{parameromorphic} if it is $\mathbb{C}'$-valued on an open dense subset of $\Sigma$ and for arbitrary $p \in \Sigma$, there exists a paraholomorphic function $\phoro{g}$ such that $\phoro{fg}$ is paraholomorphic near $p$.

\begin{rema}
	We note that the open dense subset above may not even be connected, a fact that is one of the many factors contributing to the difficulty of studying parameromorphic functions. However, in this paper, parameromorphic functions only appear as a part of Weierstrass data for the Weierstrass-type representation, where we require that a parameromorphic function $h$ must be accompanied by a paraholomorphic function $\eta$ so that $h^2\eta$ is paraholomorphic (see Fact \ref{fact:Weierstrass1}).
\end{rema}

We define a few elementary paracomplex analytic functions that are used in this paper here via analytically extending the real counterparts.
The exponential function $\phoro{e}^z$ is defined by
	\[
		\phoro{e}^z := \sum_{n = 0}^{\infty} \frac{z^n}{n!}
	\]
while the circular and hyperbolic functions are defined by
	\begin{equation}\label{eqn:cossin}
		\begin{gathered}
			\pcosh z := \sum_{n = 0}^{\infty} \frac{z^{2n}}{(2n)!},
				\quad
			\psinh z := \sum_{n = 0}^{\infty} \frac{z^{2n + 1}}{(2n + 1)!}, \\
			\pcos z := \sum_{n = 0}^{\infty} (-1)^n \frac{z^{2n}}{(2n)!},
				\quad
			\psin z := \sum_{n = 0}^{\infty} (-1)^n \frac{z^{2n + 1}}{(2n + 1)!}
		\end{gathered}
	\end{equation}
suggesting that we have the paracomplex version of Euler's formula
	\[
		\phoro{e}^{j z} = \pcosh z + j \,\psinh z
	\]
for any $z$. We also define the hyperbolic tangent and tangent functions by
	\begin{gather*}
		\ptanh z := \frac{\psinh z}{\pcosh z},
			\quad
		\ptan z := \frac{\psin z}{\pcos z}.
	\end{gather*}
Since these functions are the analytic continuations of the corresponding real hyperbolic tangent and tangent functions, $\ptanh{z}$ is defined on $\mathbb{C}'$ but $\ptan{z}$ is defined on $\{z\in\mathbb{C}' \mid |\Re{z}\pm \Im{z}|<\tfrac{\pi}{2} \}$.

\begin{rema}
	With the exception of $\ptan{z}$, the paracomplex-valued elementary functions defined above are paraholomorphic functions. $\ptan{z}$ is a parameromorphic function by our definition since $\ptan{z}\pcos{z} = \psin{z}$ is paraholomorphic.
\end{rema}

\begin{rema}
	We note here that the definitions of circular functions $\psin z$ and $\pcos z$ are different from those in \cite{konderak_weierstrass_2005}.
	In \cite{konderak_weierstrass_2005}, these functions were defined via the paracomplex exponential function and the paracomplex Euler's formula; in \eqref{eqn:cossin}, these functions are defined via analytic continuation from the real counterparts.
\end{rema}


\subsection{Timelike minimal surface theory}
Let $\mathbb{R}^{2,1}$ be the Minkowski $3$-space endowed with Lorentzian metric 
	\[
		\langle (\xi_1,\xi_2,\xi_0), (\zeta_1,\zeta_2,\zeta_0)\rangle:=\xi_1\zeta_1+\xi_2\zeta_2-\xi_0\zeta_0,
	\]
and let $\mathbb{R}^{1,1}$ denote the Minkowski $2$-plane endowed with Lorentzian metric 
	\[
		\langle (\xi_1, \xi_0), (\zeta_1, \zeta_0)\rangle_{1,1} := \xi_1\zeta_1 - \xi_0\zeta_0.
	\]
We identify the set of paracomplex numbers $\mathbb{C}'$ with Minkowski $2$-plane $\mathbb{R}^{1,1}$ via $x + j y \leftrightarrow (x, y)$, and we let $\Sigma$ denote a simply-connected domain with coordinates $(x,y)$ in $\mathbb{R}^{1,1}$.

Let $F : \Sigma \to \mathbb{R}^{2,1}$ be a timelike immersion.
As proved in \cite[p.13]{weinstein_introduction_1996}, there always exist null coordinates $(u,v)$ at each point on $\Sigma$.
Hence, Lorentz conformal coordinates $(x, y)$ also exist, by the relation
	\[
		(x,y)=\left(\frac{u+v}{2},\frac{u-v}{2}\right),
	\]
so that the induced metric $\dif s^2$ is represented as
	\begin{equation}\label{eqn:firstFundamental}
		\dif s^2 = \rho^2 (\dif x^2 - \dif y^2)  = \rho^2 \dif z \dif \bar{z} = \rho^2 \dif u \dif v
	\end{equation}
for some function $\rho : \Sigma \to \mathbb{R}_+$ with at least $C^2$-differentiability, where $\mathbb{R}_+$ is the set of positive real numbers. (Later, we will see that under the full assumptions of this paper, $\rho$ becomes analytic; see Proposition \ref{prop:solutionFG} and Theorem \ref{theo:wData}).
We choose the spacelike unit normal vector field $N : \Sigma \to \mathbb{S}^{1,1}$, where
	\[
		\mathbb{S}^{1,1} := \{\xi \in \mathbb{R}^{2,1} : \langle \xi, \xi \rangle = 1\}.
	\]

Timelike minimal surfaces inherit Lorentzian metric from the ambient space; hence, by using paracomplex analysis over the set of paracomplex numbers $\mathbb{C}'$, Konderak has shown that timelike minimal surfaces also admit a Weierstrass-type representation \cite{konderak_weierstrass_2005} (see also \cite{takahashi_tokuitenwo_2012, yasumoto_weierstrass-type_2018}):

\begin{fact}\label{fact:Weierstrass1}
	Any timelike minimal surface $F: \Sigma \subset \mathbb{C}' \to \mathbb{R}^{2,1}$ can be locally represented as
		\[
			F(x,y) = \Re\int(2h, 1 - h^2, -j(1 + h^2))\eta\,\dif z
		\]
	over a simply-connected domain $\Sigma$ on which $h$ is parameromorphic, while $\eta$ and $h^2\eta$ are paraholomorphic.
	Furthermore, the induced metric of the surface becomes
		\begin{equation}\label{eqn:wConformal}
			\dif s^2 = (1 + |h|^2)^2 |\eta|^2 (\dif x^2 - \dif y^2).
		\end{equation}
	We call $(h, \eta \dif z)$ the \emph{Weierstrass data} of the timelike minimal surface $F$.
\end{fact}

On the other hand, timelike minimal surfaces admit another representation based on null coordinates, found by McNertney \cite{mcnertney_one-parameter_1980}:
\begin{fact}\label{fact:Weierstrass2}
	Any timelike minimal surface $F$ can be locally written as the sum of two null curves $\alpha$ and $\beta$:
	\begin{equation}\label{eq:null_decomp}
		F(u,v) = \frac{\alpha(u)+\beta(v)}{2}.
	\end{equation}
	We call such $\alpha$ and $\beta$ the \emph{generating null curves} of $F$.
\end{fact}

\begin{rema}\label{rema:asso_family}
	Similar to the minimal surfaces and maximal surfaces cases, timelike minimal surfaces also admit associated families and conjugate timelike minimal surfaces:
	\begin{itemize}
		\item Given a Lorentz conformally parametrized timelike minimal surface $F$ with Weierstrass data $(h, \eta \dif z)$, we define $F^\varphi$ to be a member of the \emph{associated family of $F$} if $F^\varphi$ is given by the Weierstrass data $(h, \phoro{e}^{j \varphi} \eta \dif z)$ for some $\varphi \in \mathbb{R}$ (note that $\phoro{e}^{j \varphi} \in \mathbb{H}$, where $\mathbb{H} := \{z \in \mathbb{C}' : |z|^2 = 1 \}$).
		However, unlike the minimal surfaces and maximal surfaces cases, the \emph{conjugate timelike minimal surface} of a given timelike minimal surface is not in the associated family: the conjugate timelike minimal surface $F^*$ of $F$ is given by the Weierstrass data $(h, j \eta \dif z)$.
		\item Given a timelike minimal surface $F$ generated by null curves $\alpha(u)$ and $\beta(v)$, $F^\mu$ is a member of the associated family of $F$ if $F^\mu$ is generated by null curves $\mu\,\alpha(u)$ and $\frac{1}{\mu}\beta(v)$ for a fixed $\mu > 0$, while the conjugate timelike minimal surface of $F$ if $F^*$ is generated by null curves $\alpha(u)$ and $-\beta(v)$.
			We note that the parameters of the associated family $\varphi$ and $\mu$ are related by $e^{\varphi} = \mu$.
	\end{itemize}
\end{rema}

Following \cite{inoguchi_timelike_1998} (see also \cite{fujioka_timelike_2003, inoguchi_timelike_2004}), we define the \emph{Hopf pair} of $F$ as
	\[
		Q \dif u^2 := \langle F_{uu}, N \rangle \dif u^2, \quad R \dif v^2 := \langle F_{vv}, N \rangle \dif v^ 2
	\]
using the null coordinates $(u,v)$.
In terms of the Lorentz conformal coordinates, the \emph{Hopf differential} $\phoro{q} \dif z^2$ of $F$ can be defined from the Hopf pair of $F$ via
	\[
		\phoro{q}  \dif z^2 = Q \dif u^2 + R \dif v^2
	\]
for some paracomplex-valued function $\phoro{q}$ where $\phoro{q} = \frac{Q + R}{2} + j \frac{Q - R}{2}$.
We call a point $(x,y) \in \Sigma$ an \emph{umbilic point} of $F$ if $\phoro{q} = 0$ on $(x,y)$, and a \emph{quasi-umbilic point} of $F$ if $\phoro{q} \neq 0$ but $QR = 0$ on $(x,y)$ (see also \cite[Remark 4.3]{inoguchi_timelike_2004} or \cite[Definition 1.1]{clelland_totally_2012}).
Since the Gaussian curvature at umbilic and quasi-umbilic points vanishes, we call them \emph{flat points}.

Following \cite[Definition 3.1]{fujioka_timelike_2003} (see also \cite{fujioka_timelike_2008}), we say that $(x,y)$ are \emph{isothermic} (or conformal curvature line) coordinates of $F$ if $\phoro{q}$ is real on $\Sigma$; we say that $(x,y)$ are \emph{anti-isothermic} (or conformal asymptotic line) coordinates if $\phoro{q}$ is pure imaginary on $\Sigma$.
For a non-planar timelike minimal surface without flat points on $\Sigma$, it is known that there exist either isothermic or anti-isothermic coordinates $(x,y)$ \cite{fujioka_timelike_2003}.

\begin{rema}
	One can also characterize the existence of isothermic or anti-isothermic coordinates on any timelike minimal surface by examining the sign of the Gaussian curvature (see \cite[p.629]{magid_lorentzian_2005} or \cite[Theorem 3.4]{akamine_behavior_nodate}).
\end{rema}

Since we are interested in timelike minimal surfaces with planar curvature lines, we  assume that the mean curvature $H \equiv 0$ on the domain.
Furthermore, we require that $F$ is without flat points and has negative Gaussian curvature on its domain, so that $F$ admits isothermic coordinates.
Note that by doing this, we exclude the case when $F$ is a timelike plane as well.
Then an analogous result to \cite[Lemma 1.1]{bobenko_painleve_2000} for isothermic timelike surfaces implies that we may assume $\phoro{q} = -\frac{1}{2}$.
Calculating the Gauss-Weingarten equations then gives us
\begin{equation}\label{eqn:gaussW}
	\begin{cases}
		F_{xx} = F_{yy} = - N + \frac{\rho_x}{\rho} F_x + \frac{\rho_y}{\rho} F_y, \\
		F_{xy} = \frac{\rho_y}{\rho} F_x + \frac{\rho_x}{\rho} F_y, \\
		N_{x} = \frac{1}{\rho^2} F_x, \quad N_{y} = -\frac{1}{\rho^2} F_y,
	\end{cases}
\end{equation}
while the Gauss equation (or the integrability condition) becomes
\[
	\rho \cdot \Box\rho - ({\rho_x}^2 - {\rho_y}^2) -1 = 0,
\]
where $\Box := \partial_x^2 - \partial_y^2$ is the d'Alembert operator.

\subsection{Planar curvature line condition and the analytic classification}
We now calculate the condition the Lorentz conformal factor $\rho$ must satisfy for a timelike minimal surface $F$ to have planar curvature lines.
\begin{lemm}
For a timelike minimal surface with no flat points, the following statements are equivalent:
	\begin{enumerate}
		\item $x$-curvature lines are planar.
		\item $y$-curvature lines are planar.
		\item $\rho_{xy}= 0$.
	\end{enumerate}
\end{lemm}

\begin{proof}
Similar to the proof of \cite[Lemma 2.1]{cho_deformation_2017} and \cite[Lemma 2.1]{cho_deformation_2018}, we show this by calculating $\det(F_x, F_{xx}, F_{xxx}) = 0$ and $\det(F_y, F_{yy}, F_{yyy}) = 0$ using \eqref{eqn:gaussW}. 
\end{proof}

Therefore, by finding solutions to the following system of partial differential equations, we may find all timelike minimal surfaces with planar curvature lines:
	\begin{subnumcases}{\label{eqn:pde}}
		\rho \cdot \Box\rho - ({\rho_x}^2 - {\rho_y}^2) -1 = 0\label{eqn:pde1} &\text{(timelike minimality condition),}\\
		\rho_{xy}= 0 \label{eqn:pde2} &\text{(planar curvature line condition)}.
	\end{subnumcases}
To solve the above system, we first reduce \eqref{eqn:pde} to a system of ordinary differential equations as in \cite[Theorem 2.1]{abresch_constant_1987}.

\begin{lemm}\label{lemm:solutionRho}
For a solution $\rho: \Sigma \to \mathbb{R}_+$ to \eqref{eqn:pde}, there exist real-valued functions $f(x)$ and $g(y)$ such that
	\begin{subnumcases}{\label{eqn:recover}}
		\rho_x = f(x) \label{eqn:recover1}, \\
		\rho_y = g(y) \label{eqn:recover2}.
	\end{subnumcases}
Then, $\rho$ can be written in terms of $f(x)$ and $g(y)$ as follows:
\begin{description}
	\item[Case (1)] If $\Box\rho$ is nowhere zero on $\Sigma$, then
			\begin{equation} \label{eqn:solutionRho1}
				\rho(x,y) = \frac{f(x)^2 - g(y)^2 + 1}{f_x(x) - g_y(y)},
			\end{equation}
		where $f(x)$ and $g(y)$ satisfy the following system of ordinary differential equations:
			\begin{subnumcases}{\label{eqn:ode}}
				(f_x(x))^2 = (c - d) f(x)^2 + c \label{eqn:ode1}\\
				f_{xx}(x)= (c - d) f(x) \label{eqn:ode2}\\
				(g_y(y))^2 = (c - d) g(y)^2 + d \label{eqn:ode3}\\
				g_{yy}(y)= (c - d) g(y) \label{eqn:ode4}
			\end{subnumcases}
		for real constants $c$ and $d$ such that $c^2 + d^2 \neq 0$.

	\item[Case (2)] If $\Box\rho \equiv 0$ on $\Sigma$, i.e.\ $\Box\rho$ is identically zero on $\Sigma$, then
		\begin{equation} \label{eqn:solutionRho2}
			\rho(x,y) = (\sinh \phi) \cdot x - (\cosh \phi) \cdot y
		\end{equation}
		where $f(x) = \sinh \phi$ and $g(y) = - \cosh \phi$ for some constant $\phi \in \mathbb{R}$.
	\end{description}
\end{lemm}
\begin{proof}
Arguments for the proof of this lemma mirror those in the proof of \cite[Lemma 2.2]{cho_deformation_2017} and \cite[Lemma 2.2]{cho_deformation_2018}: we here only give an outline of the proof.

First, \eqref{eqn:recover} can be shown directly from \eqref{eqn:pde2}. Now if we assume that $\Box\rho$ is not identically equal to zero, then there is a point $(x_0, y_0)$ such that $\Box\rho(x_0, y_0) \neq 0$, suggesting that we can choose a neighborhood $\Sigma \subset \mathbb{R}^{1,1}$ of $(x_0, y_0)$ such that $\Box\rho$ is nowhere zero on $\Sigma$. On such $\Sigma$, we can directly show the remaining claims using \eqref{eqn:pde1} and \eqref{eqn:recover}.

On the other hand, if $\Box\rho$ is identically equal to zero on some simply-connected domain $\Sigma$, then \eqref{eqn:pde1} and \eqref{eqn:recover} imply that ${\rho_x}^2 - {\rho_y}^2 = f(x)^2 - g(y)^2 =  -1$; hence, $f(x)$ and $g(y)$ are constant functions.
\end{proof}

We now solve \eqref{eqn:ode} by first obtaining a general solution, and then finding an appropriate initial condition to get an explicit solution for $f(x)$ and $g(y)$.
First, if $c = d$, then \eqref{eqn:ode1} and \eqref{eqn:ode3} imply that $c = d > 0$, and using \eqref{eqn:ode2} and \eqref{eqn:ode4}, we may obtain the explicit solutions:
	\begin{equation}\label{eqn:cequalsd}
		f(x) = \pm \sqrt{c}\,x + \tilde{C}_1\quad\text{and}\quad g(y) = \pm \sqrt{d}\,y + \tilde{C}_2
	\end{equation}
for some real constants of integration $\tilde{C}_1$ and $\tilde{C}_2$.

Now, assuming that $c \neq d$, we can explicitly solve for $f(x)$ and $g(y)$ to find that
	\begin{equation}\label{eqn:cneqd}
		\begin{aligned}
		f(x) &= C_1 e^{\sqrt{c - d}\,x} + C_2 e^{-\sqrt{c - d}\,x}, \quad 4(d - c) C_1 C_2 = c,\\
		g(y) &= C_3 e^{\sqrt{c - d}\,y} + C_4 e^{-\sqrt{c - d}\,y}, \quad 4(d - c) C_3 C_4 = d,
		\end{aligned}
	\end{equation}
where $C_1, \ldots, C_4 \in \mathbb{C}$ are constants of integration.
Furthermore, since $f(x)$ and $g(y)$ are real-valued functions, $C_1, \ldots, C_4$ must satisfy
	\begin{equation}\label{eqn:constantsIntegration}
		\begin{cases}
		C_1, C_2, C_3, C_4 \in \mathbb{R}, &\text{if }c > d,\\
		C_1 = \overline{C_2} \text{ and } C_3 = \overline{C_4}, &\text{if }d > c,
		\end{cases}
	\end{equation}
where $\bar{\cdot}$ denotes the usual complex conjugation.

In the following series of lemmata, we identify the correct initial conditions based on the values of $c$ and $d$.

\begin{lemm}[cf.\ Lemma 2.3 of \cite{cho_deformation_2018}]\label{lemm:zero}
	$f(x)$ \resp{$g(y)$} satisfying \eqref{eqn:ode} has a zero if and only if either $c > 0$ or $f(x) \equiv 0$ \resp{$d > 0$ or $g(y) \equiv 0$}.
\end{lemm}

\begin{lemm}[cf.\ Lemma 2.4 of \cite{cho_deformation_2018}]
	$f(x)$ \resp{$g(y)$} satisfying \eqref{eqn:ode} has no zero if and only if either $c < 0$ or $f(x) = \pm e^{\sqrt{-d} x}$, where $d < 0$ \resp{$d < 0$ or $g(y) = \pm e^{\sqrt{c} y}$ where $c > 0$}.
\end{lemm}

Therefore, we can conclude the following about the nature of $f(x)$ and $g(y)$ depending on the values of $c$ and $d$:
	\begin{equation}\label{eqn:16cases}
	\begin{array}{l || l}
		\begin{aligned}
			c > 0 &: f\text{ has a zero} \\
			c = 0 &: \begin{cases} f \equiv 0 &(c = 0_0)\\ f = \pm e^{\sqrt{-d}\,x}, d < 0 &(c = 0_e)  \end{cases} \\
			c < 0 &: f\text{ has no zero} 
		\end{aligned}
		\quad
		&
		\quad
		\begin{aligned}
			d > 0 &: g\text{ has a zero}\\
			d = 0&: \begin{cases} g \equiv 0 &(d = 0_0)\\ g = \pm e^{\sqrt{c}\,y}, c > 0 &(d = 0_e)  \end{cases} \\
			d < 0 &: g\text{ has no zero}.
		\end{aligned}
	\end{array}
	\end{equation}
For the cases where $f$ or $g$ have no zero, we use the following lemmata to identify a possible initial condition.
\begin{lemm}\label{lemm:gOne}
	Suppose that $g(y) \not\equiv 0$, i.e.\ $g(y)$ is not identically equal to $0$. Then there is some $y_0$ such that $g(y_0)^2 = 1$ if and only if $c \geq 0$.
\end{lemm}
\begin{proof}
	If $c = d$, then the statement is a direct result of \eqref{eqn:cequalsd}; therefore, assume that $c \neq d$.
	To show one direction, assume that there is some $y_0$ such that $g(y_0)^2 = 1$.
	Then \eqref{eqn:ode3} implies that $c \geq 0$.

	Now assume that $c \geq 0$. If $c > d$, then for
		\[
			y_0 := \frac{\log\left(\frac{\sqrt{c - d} + \sqrt{c}}{2|C_3|\sqrt{c - d}}\right)}{\sqrt{c - d}},
		\]
	we have that $g(y_0)^2 = 1$ via \eqref{eqn:cneqd}.
	
	If $c < d$, then from \eqref{eqn:constantsIntegration}, we have that $\overline{C_3}C_3 = \frac{d}{4(d - c)}$, implying that we may write $C_3 = \sqrt{\tfrac{d}{4(d-c)}}e^{i\Theta}$ and $C_4 = \sqrt{\tfrac{d}{4(d-c)}}e^{-i\Theta}$
	for some $\Theta \in \mathbb{R}$. Therefore, by \eqref{eqn:cneqd}, we have that
		\[
			g(y) = \sqrt{\tfrac{d}{d - c}} \cos \left(\sqrt{d - c} \, y + \Theta\right).
		\]
	Since we have $d > c \geq 0$, we have that $\sqrt{\tfrac{d}{d - c}} > 1$; therefore, there is some $y_0$ such that $g(y_0)^2 = 1$.
\end{proof}

\begin{lemm}
	Suppose that $f(x) \not\equiv 0$. Then there is some $x_0$ such that $f(x_0)^2 = 1$ if and only if $2c \geq d$.
\end{lemm}
\begin{proof}
	The proof is similar to that of Lemma \ref{lemm:gOne}.
\end{proof}

\begin{lemm}
	If $c < 0$ and $d < 0$, then there is some $x_0$ \resp{$y_0$} such that $f_x(x_0) = 0$ \resp{$g_y(y_0) = 0$}.
\end{lemm}
\begin{proof}
	From \eqref{eqn:ode1} and $c < 0$, we deduce that $c - d > 0$.
	Therefore, if
		\[
			x_0 := \frac{\log\left(\frac{-c}{4(c - d) C_1^2}\right)}{2 \sqrt{c - d}},
		\]
	then $f_x(x_0) = 0$ by \eqref{eqn:cneqd}.
	The statement for $g(y)$ is proven similarly.
\end{proof}

Therefore, of the possible 16 cases coming from \eqref{eqn:16cases}, we only need to consider the 8 cases specified in Table \ref{table:sheets} with their respective initial conditions. Note that we shift the parameters $x$ and $y$ to assume without loss of generality that $x_0 = y_0 = 0$.
\begin{table}
	\begin{tabular}{@{} lll @{}}
		\toprule
			 & Initial condition & Applicable values of $(c, d)$\\
		\midrule
			Case (1a) & $f(0) = 0$, $g(0) = 0$ & $(+, +)$, $(+, 0_0)$, $(0_0, +)$ \\
			Case (1b) & $f(0) = 0$, $g(0) = \pm 1$ & $(+, 0_e)$, $(+, -)$, $(0_0, -)$\\
			Case (1c) & $f(0) = \pm 1$, $g(0) = \pm 1$ & $(0_e, -)$\\
			Case (1d) & $f_x(0) = 0$, $g_y(0) = 0$ & $(-, -)$ \\
		\bottomrule
	\end{tabular}
	\caption{Choice of initial conditions and the corresponding applicable cases.}
	\label{table:sheets}
\end{table}

By using these initial conditions to solve \eqref{eqn:ode}, we obtain the following set of explicit solutions for $f$ and $g$.
\begin{prop}\label{prop:solutionFG}
	For a non-planar generalized timelike minimal surface with planar curvature lines $F(x,y)$, the real-analytic solution $\rho: \mathbb{R}^{1,1} \to \mathbb{R}$ of \eqref{eqn:pde} is precisely given as follows:
	\begin{description}
		\item[Case (1)] Let $\Box \rho \not\equiv 0$, i.e.\ $\Box \rho$ is not identically equal to zero. Then,
			\[
				\rho(x,y) = \frac{f(x)^2 - g(y)^2 + 1}{f_x(x) - g_y(y)},
			\]
		where $f(x)$ and $g(y)$ are given as follows (see Table \ref{table:sheets}):
		\begin{description}
			\item[Case (1a)] For $c \geq 0$ and $d \geq 0$ such that $c^2 + d^2 \neq 0$,
				\begin{align*}
					f(x) &=
						\begin{cases}
            						\sqrt{\frac{c}{c - d}} \sinh{(\sqrt{c - d}\, x)}, &\text{if }c \neq d,\\
							\sqrt{c}\, x, &\text{if } c = d,
						\end{cases}\\
	            			g(y) &=
						\begin{cases}
							-\sqrt{\tfrac{d}{c - d}} \sinh{(\sqrt{c - d}\,y)}, &\text{if } c \neq d,\\
							- \sqrt{d} \, y, &\text{if } c = d.
						\end{cases}
				\end{align*}
			\item[Case (1b)] For $c \geq 0$ and $d \in \mathbb{R}$ such that $c^2 + d^2 \neq 0$,
				\begin{align*}
					f(x) &= \sqrt{\tfrac{c}{c - d}} \sinh{(\sqrt{c - d}\, x)},\\
	            			g(y) &= -\cosh{(\sqrt{c - d}\, y)}-\sqrt{\tfrac{c}{c - d}} \sinh{(\sqrt{c - d}\,y)}.
				\end{align*}
			\item[Case (1c)] For $c \geq 0$ and $d < 2c$,
				\begin{align*}
					f(x) &= \cosh{(\sqrt{c - d}\, y)} + \sqrt{\tfrac{2c - d}{c - d}} \sinh{(\sqrt{c - d}\,y)},\\
	            			g(y) &= -\cosh{(\sqrt{c - d}\, y)}-\sqrt{\tfrac{c}{c - d}} \sinh{(\sqrt{c - d}\,y)}.
				\end{align*}
			\item[Case (1d)] For $d< c < 0$,
				\begin{align*}
					f(x) &= \sqrt{\tfrac{c}{d - c}} \cosh{(\sqrt{c - d}\, x)},\\
	            			g(y) &= -\sqrt{\tfrac{d}{d - c}} \cosh{(\sqrt{c - d}\, y)}.
				\end{align*}
		\end{description}
		\item[Case (2)] If $\Box \rho \equiv 0$, then for some constant $\phi$ such that $\phi \in \mathbb{R}$,
			\[
				\rho(x,y) = (\sinh \phi) \cdot x - (\cosh \phi) \cdot y.
			\]
\end{description}
\end{prop}
\begin{proof}
The proof is essentially the same as that of \cite[Proposition 2.1]{cho_deformation_2018}.
\end{proof}

\begin{rema}\label{rema:global}
We make a few essential remarks about Proposition \ref{prop:solutionFG}.
	\begin{itemize}
		\item We have now extended the domain globally under our hypotheses.
			Therefore, we may deduce that non-planar timelike minimal surfaces with planar curvature lines do not have any flat points globally, and we may drop this condition from now.
			(In fact, we may also infer that these surfaces admit isothermic coordinates globally.)
		\item We now allow $\rho$ to map into $\mathbb{R}$ as opposed to $\mathbb{R}_+$. By doing so, we now treat timelike minimal surfaces with planar curvature lines as generalized timelike minimal surfaces.
			(We can show that these surfaces are actually \emph{minfaces}, see Section \ref{sect:singularities} in the appendix.)
		\item In cases (1a) through (1c), we allow $c - d < 0$. Even in such case, we see that $f(x)$ and $g(y)$ are real-valued analytic functions via the identities
			\[
				\cosh{(\sqrt{c-d}\, x)} = \cos{(\sqrt{d-c}\, x)}, \quad \sinh{(\sqrt{c-d}\, x)} = \sqrt{-1} \sin{(\sqrt{d-c}\, x)}.
			\]
			Furthermore, cases (1b) and (1c) also include the case when $c = d$.
			However, since the resulting solution is the same solution as that in case (1a) up to shift of parameters $x$ and $y$, we do not write these cases explicitly.
		\item For case (2), we note that this is a Lorentzian analogue of the Bonnet-Lie transformation (see, for example, \cite[\S 394]{bianchi_lezioni_1903}), giving an associated family of the surface with solution $\rho(x,y) = -y$ up to coordinate change. To see this explicitly, we introduce a parameter $\lambda$ and consider the following change of coordinates:
			\[\begin{cases}
				\tilde{x} := \cosh\phi \cdot x - \sinh\phi \cdot y, & \tilde{y}:= \sinh\phi \cdot x - \cosh\phi \cdot y, \\
				\lambda:=\phoro{e}^{-j\phi}, &\tilde{\mathtt{q}}:=-\frac{1}{2}\lambda^{-2}=\lambda^{-2}\texttt{q}.
			\end{cases}\]
	\end{itemize}
\end{rema}

Summarizing, we obtain the following complete classification of non-planar timelike minimal surfaces with planar curvature lines.
\begin{theo}\label{theo:bifurcation}
	Let $F(x,y)$ be a non-planar generalized timelike minimal surface in $\mathbb{R}^{2,1}$ with isothermic coordinates $(x,y)$ such that the induced metric is $\dif s^2 = \rho^2(\dif x^2 - \dif y^2)$.
	Then $F$ has planar curvature lines if and only if $\rho(x,y)$ satisfies Proposition \ref{prop:solutionFG}.
	Furthermore, for different values of $(c, d)$ or $\lambda$ as in Remark \ref{rema:global}, the Lorentz conformal factor $\rho(x,y)$ or the surface $F(x,y)$ has the following properties, based on Figure \ref{fig:bifurcation}:

\begin{description}
	\item[Case (1)] If $\Box \rho \not \equiv 0$, when $(c,d)$ are on the region marked by
		\begin{itemize}
			\item \circled{$1$} $:$ $\rho$ is constant in the $x$-direction, but periodic in the $y$-direction,
			\item \circled{$2$} $:$ $\rho$ is periodic in both the $x$-direction and the $y$-direction,
			\item \circled{$3$}, \circled{$4$}, \circled{$6$}, \circled{$7$}, \circled{$9$}, or \circled{$10$} $:$ $\rho$ is not periodic in both the $x$-direction and the $y$-direction,
			\item \circled{$5$} $:$ $\rho$ is not periodic in the $x$-direction, but constant in the $y$-direction,
			\item \circled{$8$} $:$ $\rho$ is constant in the $x$-direction, but not periodic in the $y$-direction.
		\end{itemize}

	\item[Case (2)] If $\Box \rho \equiv 0$, when $\lambda$ is on the region marked by
		\begin{itemize}
			\item \circled{$11$} $:$ $F$ is a surface of revolution,
			\item \circled{$12$} $:$ $F$ is a surface in the associated family of \circled{$11$}.
		\end{itemize}
\end{description}
\end{theo}

\savebox{\mybox}{\includegraphics[width=0.20\textwidth]{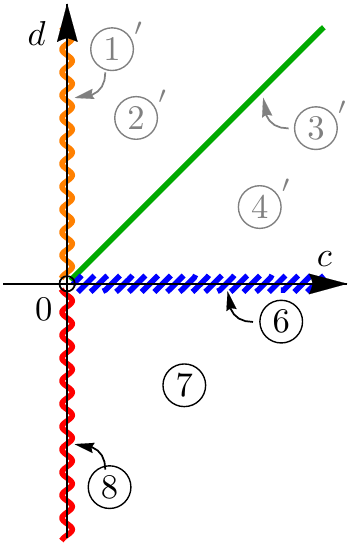}}
\begin{figure}
	\centering
	\begin{subfigure}{0.20\textwidth}
		\vbox to \ht\mybox{%
			\vfill
			\includegraphics[width=\textwidth]{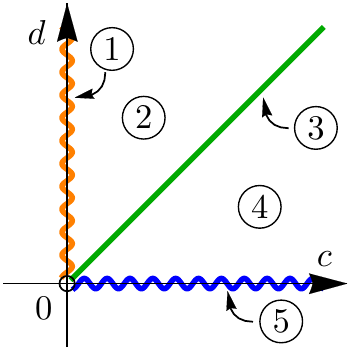}
			\vfill
		}
		\caption{Case (1a)}
		\label{fig:sheet1}
	\end{subfigure}
	\quad
	\begin{subfigure}{0.20\textwidth}
		\usebox{\mybox}
		\caption{Case (1b)}
		\label{fig:sheet2}
	\end{subfigure}
	\quad
	\begin{subfigure}{0.20\textwidth}
		\vbox to \ht\mybox{%
			\vfill
			\includegraphics[width=\textwidth]{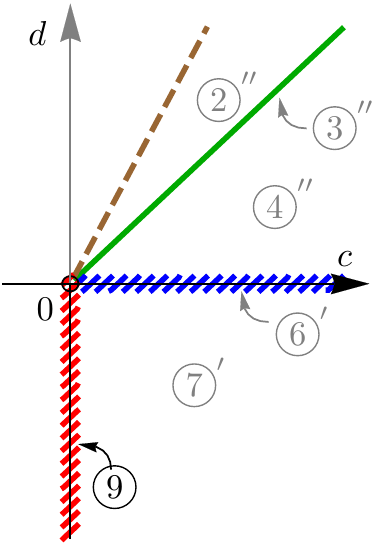}
			\vfill
		}
		\caption{Case (1c)}
		\label{fig:sheet3}
	\end{subfigure}
	\quad
	\begin{subfigure}{0.20\textwidth}
		\vbox to \ht\mybox{%
			\vfill
			\includegraphics[width=\textwidth]{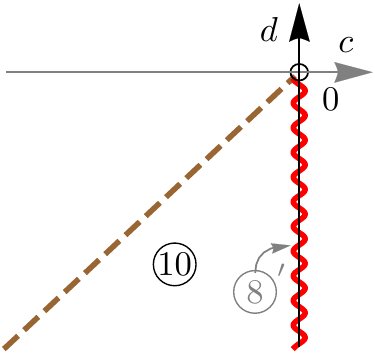}
			\vfill
		}
		\caption{Case (1d)}
		\label{fig:sheet4}
	\end{subfigure}
	\par\bigskip
	\savebox{\mybox}{\includegraphics[width=0.19\textwidth]{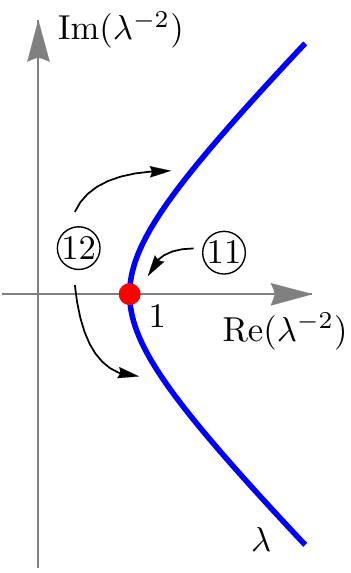}}
	\begin{subfigure}{0.19\textwidth}
		\usebox{\mybox}
		\caption{Case (2)}
		\label{fig:case2}
	\end{subfigure}
	\quad\quad
	\begin{subfigure}{0.4\textwidth}
		\begin{center}
		\vbox to \ht\mybox{%
			\vfill
		\begin{tabular}[c]{@{} l l l @{}}
			\toprule
				& values of $(c, d)$ & represents\\
			\midrule
				\includegraphics[scale=0.5]{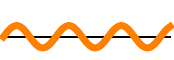} & $c = 0, d > 0$ & surface \circled{$1$} \\
				\includegraphics[scale=0.5]{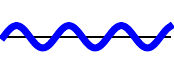} & $c > 0, d = 0$ & surface \circled{$5$} \\
				\includegraphics[scale=0.5]{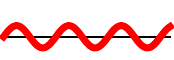} & $c = 0, d < 0$ & surface \circled{$8$} \\
				\includegraphics[scale=0.5]{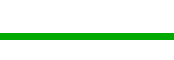} & $c = d$ & surface \circled{$3$} \\
				\includegraphics[scale=0.5]{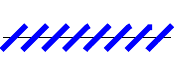} & $c > 0, d = 0$ & surface \circled{$6$} \\
				\includegraphics[scale=0.5]{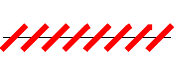} & $c = 0, d < 0$ & surface \circled{$9$} \\
				\includegraphics[scale=0.5]{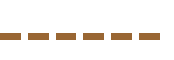} &  & boundary\\
			\bottomrule
		\end{tabular}
			\vfill
		}		
		\end{center}
	\end{subfigure}
	\caption{Bifurcation diagrams per choice of initial conditions and values of $(c,d)$ (see also Table \ref{table:sheets}). For example, surface \circlenum{8} is obtained by choosing initial conditions $f(0) = 0$ and $g(0) = \pm 1$ (since it is under case (1b), see Table \ref{table:sheets}) and choosing $(c,d)$ such that $c = 0, d < 0$. On the other hand, surface \circlenum{9} is obtained by choosing initial conditions $f(0) = \pm 1$ and $g(0) = \pm 1$ (since it is under case (1c)) and letting $c = 0, d < 0$. Finally, $\circlenum{2}'$ gives the same surface as $\circlenum{2}$ up to shift of parameters.}
	\label{fig:bifurcation}
\end{figure}

\subsection{Axial directions and the Weierstrass data}
From the explicit solutions of the Lorentz conformal factor $\rho$, we now aim to recover the Weierstrass data.
The Weierstrass data are not unique for a given timelike minimal surface; for example, applying any rigid motion to the surface will change its Weierstrass data.
Therefore, to decide how the surface is aligned in the ambient space $\mathbb{R}^{2,1}$, we use the existence of \emph{axial directions} as defined in \cite[Proposition 2.2]{cho_deformation_2017} (see also \cite[Proposition 3.A]{walter_explicit_1987}).
After aligning axial directions according to its causality, we recover the unit normal vector, allowing us to calculate the Weierstrass data.
First, we show the existence of axial directions.

\begin{prop}\label{prop:axial}
If there exists $x_1$ \resp{$y_1$} such that $f(x_1) \neq 0$ \resp{$g(y_1) \neq 0$} in Proposition \ref{prop:solutionFG}, then there exists a unique non-zero constant vector $\vec{v}_1$ \resp{$\vec{v}_2$} such that
	\[
		\langle m(x,y), \vec{v}_1\rangle = \langle m_y(x,y), \vec{v}_1 \rangle = 0
			\quad\text{\resp{$\langle n(x,y), \vec{v}_2\rangle = \langle n_x(x,y), \vec{v}_2 \rangle = 0$}},
	\]
where $m := \rho^{-2} (F_x \times F_{xx})$ \resp{$n := \rho^{-2} (F_y \times F_{yy})$} and
\begin{equation}\label{eqn:v1}
	\begin{gathered}
	\vec{v}_1 := -\frac{\rho_x}{\rho}N - \frac{\rho_{xx}\rho - {\rho_x}^2}{\rho^2}F_x + \frac{\rho_x \rho_y}{\rho^2} F_y\quad
		\text{\resp{$\vec{v}_2 := \frac{\rho_y}{\rho}N - \frac{\rho_x \rho_y}{\rho^2} F_x + \frac{\rho_{yy}\rho - {\rho_y}^2}{\rho^2}F_y$}}.
	\end{gathered}
\end{equation}
Furthermore, if $\vec{v}_1$ and $\vec{v}_2$ both exist, then they are orthogonal to each other. We call $\vec{v}_1$ and $\vec{v}_2$ the \emph{axial directions} of $F(x,y)$.
\end{prop}

\begin{proof}
Similar to the proof of \cite[Proposition 2.2]{cho_deformation_2017} and \cite[Proposition 2.2]{cho_deformation_2018}, using \eqref{eqn:gaussW} and \eqref{eqn:pde}, we may calculate that all the required property holds.
\end{proof}

We use \eqref{eqn:gaussW}, \eqref{eqn:pde}, \eqref{eqn:solutionRho1}, \eqref{eqn:ode}, and \eqref{eqn:v1} to calculate that the causality of $\vec{v_1}$ and $\vec{v_2}$ depends on $c$ and $d$, respectively; explicitly,
	\[
		\langle \vec{v}_1, \vec{v}_1 \rangle = c \quad\text{and}\quad \langle \vec{v}_2, \vec{v}_2 \rangle = -d.
	\]
Hence, we remark that, by Table \ref{table:sheets}, at least one of $\vec{v}_1$ or $\vec{v}_2$ is always spacelike when they both exist. By aligning the axial directions in the ambient space $\mathbb{R}^{2,1}$ correctly, we now calculate the unit normal vector using the following lemma. Note that we define $\vec{e}_j$ as the unit vectors in the $\xi_j$ direction for $j = 1, 2, 0$.

\begin{lemm}\label{lemm:normal}
	For the different alignments of $\vec{v}_1$ or $\vec{v}_2$, we can deduce the following regarding the unit normal vector $N(x,y) = (N_1(x,y), N_2(x,y), N_0(x,y))$:
	\begin{center}
		{\renewcommand{\arraystretch}{1.2}
		\begin{tabular}{@{} ll @{}}
			\toprule
				Alignment of axial direction & Property of the unit normal vector\\
			\midrule
				$\vec{v}_1 \parallel \vec{e}_2$ & $N_2 = \pm \frac{1}{\sqrt{c}}\frac{\rho_x}{\rho}$\\
				$\vec{v}_1 \parallel  \vec{e}_1 + \vec{e}_0$ & $N_1 - N_0 = \pm\frac{\rho_x}{\rho}$\\
				$\vec{v}_1 \parallel \vec{e}_0$ & $N_0  = \pm \frac{1}{\sqrt{-c}}\frac{\rho_x}{\rho}$\\
			\midrule
				$\vec{v}_2 \parallel a_1\vec{e}_1 + a_0\vec{e}_0$ & $a_1N_1 - a_0N_0 = \pm \sqrt{\frac{a_0^2 - a_1^2}{d}} \frac{\rho_y}{\rho}$\\
				$\vec{v}_2 \parallel a_1\vec{e}_1 + a_2\vec{e}_2$ & $a_1N_1 + a_2N_2 = \pm \sqrt{\frac{a_1^2 + a_2^2}{-d}} \frac{\rho_y}{\rho}$\\
			\bottomrule
		\end{tabular}}
	\end{center}
	Here, $a_1$, $a_2$ and $a_0$ are any real constants.
\end{lemm}
\begin{proof}
The proof is similar to the proof in \cite[Proposition 2.3]{cho_deformation_2017}, and \cite[Lemma 2.7, Lemma 2.8]{cho_deformation_2018}.
\end{proof}

Using the fact that the meromorphic function $h$ of the Weierstrass data is the unit normal vector function under the stereographic projection, and that $\phoro{q} = -h_z \eta = -\frac{1}{2}$, we recover the Weierstrass data via
	\[
		h(z) = h(x,y) = \frac{1}{1 - N_1}(N_2 + j N_0) \quad\text{and}\quad \eta(z) = \frac{1}{2 h_z},
	\]
where the signs of $N_1$, $N_2$, and $N_0$ are decided so that $h$ satisfies the Cauchy-Riemann type conditions \eqref{eqn:cauchyRiemann}.

\subsubsection{Case (1a)}\label{sssec:sheet1}

Assume that $c > 0$ and $d > 0$. We have that $\vec{v}_1$ is spacelike, while $\vec{v}_2$ is timelike; therefore, we align the axial directions so that $\vec{v}_1 \parallel \vec{e}_2$ and $\vec{v}_2 \parallel \vec{e}_0$.
Then by Lemma \ref{lemm:normal}, we have that
	\[
		N = \left(\pm\sqrt{1 - \frac{1}{c}\frac{\rho_x^2}{\rho^2} + \frac{1}{d}\frac{\rho_y^2}{\rho^2}},\, \pm \frac{1}{\sqrt{c}}\frac{\rho_x}{\rho},\, \pm \frac{1}{\sqrt{d}}\frac{\rho_y}{\rho}\right).
	\]
Since we know that a homothety in the $(c,d)$-plane amounts to a homothety in the $(x,y)$-plane, by Proposition \ref{prop:solutionFG}, we can let $c = 4\cos^2 c_1$ and $d = 4\sin^2 c_1$ for $c_1 \in \left(0, \frac{\pi}{2}\right)$ without loss of generality (see Figure \ref{fig:pathS1}).
Using the unit normal vector, we find that
	\begin{equation}\label{eqn:wData1}
		\begin{aligned}
			h_1^{c_1}(z) &= \begin{cases}
					\frac{\sqrt{\cos{(2 c_1)}}}{\cos c_1 - \sin c_1} \ptanh\left(\sqrt{\cos{(2 c_1)}} \,z\right), &\text{if $c_1 \in (0, \frac{\pi}{4})$},\\
					\sqrt{2} z, & \text{if $c_1 = \frac{\pi}{4}$},\\
					\mathrlap{-\frac{\sqrt{-\cos{(2 c_1)}}}{\cos c_1 - \sin c_1} \ptan\left(\sqrt{-\cos{(2 c_1)}} \,z\right),}\hphantom{\frac{1}{2(\cos c_1 + \sin c_1)} \pcos^2 \left(\sqrt{-\cos{(2 c_1)}} \,z\right),} &\text{if $c_1 \in (\frac{\pi}{4},\frac{\pi}{2})$},
				\end{cases}\\
			\eta_1^{c_1}(z) &= \begin{cases}
					\frac{1}{2(\cos c_1 + \sin c_1)} \pcosh^2 \left(\sqrt{\cos{(2 c_1)}} \,z\right), &\text{if $c_1 \in (0, \frac{\pi}{4})$},\\
					\frac{1}{2\sqrt{2}}, & \text{if $c_1 = \frac{\pi}{4}$},\\
					\frac{1}{2(\cos c_1 + \sin c_1)} \pcos^2 \left(\sqrt{-\cos{(2 c_1)}} \,z\right), &\text{if $c_1 \in (\frac{\pi}{4},\frac{\pi}{2})$}.
				\end{cases}
		\end{aligned}
	\end{equation}
Note that $h_1^{c_1}(z)$ and $\eta_1^{c_1}(z)$ is also well-defined when $c_1 = 0, \frac{\pi}{2}$ by considering the directional limits.


\begin{rema}\label{rema:dpWdata}
The Weierstrass data given in \eqref{eqn:wData1} show that surfaces in case (1a) form a one-parameter family of surfaces.
However, by considering these surfaces separately, one can get different, and perhaps simpler, Weierstrass data.
\begin{itemize}
	\item For surfaces $\circled{1}$ and $\circled{2}$, by using $c = 4 \sinh^2 (\log \tilde{c}_1)$ and $d = 4 \cosh^2 (\log \tilde{c}_1)$ for $ \tilde{c}_1 \geq 1$, we obtain
		\[
			h_1^{\tilde{c}_1}(z) = \tilde{c}_1 \ptan z, \quad \eta_1^{\tilde{c}_1}(z) = \frac{1}{2 \tilde{c}_1}\pcos^2 z.
		\]
	\item For the surface $\circled{5}$, by letting $\vec{v}_1 \parallel \vec{e}_1$ and $\vec{v}_2 \parallel \vec{e}_0$, we obtain that
		\[
			\tilde{h}_1^{c_1}(z)\big |_{c_1 = \frac{\pi}{2}} = \phoro{e}^{z}, \quad \tilde{\eta}_1^{c_1}(z)\big |_{c_1 = \frac{\pi}{2}} = \frac{1}{2}\phoro{e}^{-z}.
		\]
\end{itemize}
\end{rema}
\begin{figure}
	\centering
	\savebox{\mybox}{\includegraphics[width=0.2\textwidth]{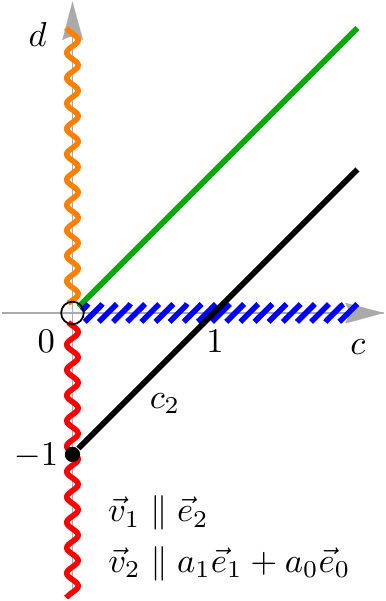}}
	\begin{subfigure}{0.2\textwidth}
		\vbox to \ht\mybox{%
			\vfill
			\includegraphics[width=\textwidth]{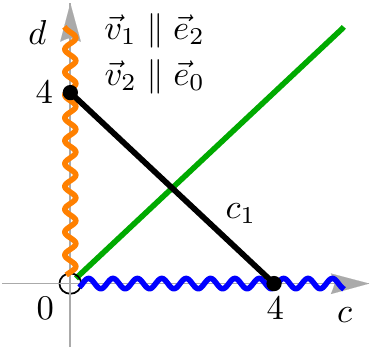}
			\vfill
		}
		\caption{Case (1a)}
		\label{fig:pathS1}
	\end{subfigure}
	\qquad
	\begin{subfigure}{0.2\textwidth}
		\usebox{\mybox}
		\caption{Case (1b)}
		\label{fig:pathS2}
	\end{subfigure}
	\qquad
	\begin{subfigure}{0.2\textwidth}
		\vbox to \ht\mybox{%
			\vfill
			\includegraphics[width=\textwidth]{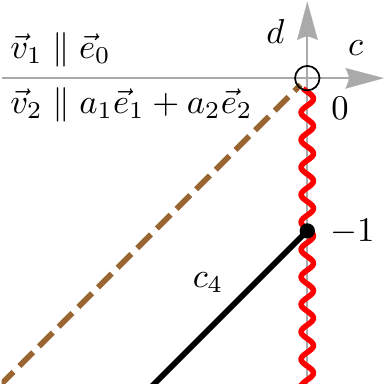}
			\vfill
		}
		\caption{Case (1d)}
		\label{fig:pathS4}
	\end{subfigure}
	\caption{$(c,d)$-paths for different cases. For the meaning of the diagram, see Figure \ref{fig:bifurcation}.}
	\label{fig:path1}
\end{figure}
	
\subsubsection{Case (1b)}\label{sssec:sheet2}
Assume that $c \geq 0$ but $d \in \mathbb{R}$, implying that now $\vec{v}_2$ changes its causal character.
Therefore, we align the axial directions so that $\vec{v}_1 \parallel \vec{e}_2$ and $\vec{v}_2 \parallel a_1\vec{e}_1 + a_0 \vec{e}_0$.
Since we only need to find the unit normal vector of surfaces $\circled{6}$, $\circled{7}$, and $\circled{8}$, we let $c = c_2^2$ and $d = c_2^2 -1$ for $c_2 \geq 0$, and further assume that $a_1 = 1$ and $a_0 = c_2$  (see Figure \ref{fig:pathS2}).
Then we have that the unit normal vector is
	\[
		N = \left(c_2 N_0 \pm \frac{\rho_y}{\rho},\, \pm \frac{1}{c_2}\frac{\rho_x}{\rho},\, N_0\right),
	\]
where $N_0$ can be found from the fact that $\langle N, N \rangle = 1$.
From the unit normal vector, after applying a shift of parameter $y \mapsto y - \log(1 + c_2)$, we calculate that
	\begin{equation}\label{eqn:wData2}
		h_2^{c_2}(z) = j \phoro{e}^{j z} - j c_2, \quad
		\eta_2^{c_2}(z) = \frac{1}{2}\phoro{e}^{-j z}.
	\end{equation}
Similar to the preceding case, note that $h_2^{c_2}(z)$ and $\eta_2^{c_2}(z)$ is also well-defined when $c_2 = 0$ by considering the directional limits.

\begin{rema}
	Note that if $c_2 > 1$, then \eqref{eqn:wData2} describes Weierstrass data for the surface $\circled{4}$, aligned differently in the ambient space $\mathbb{R}^{2,1}$ to the one given by \eqref{eqn:wData1} for $c_1 \in \left(\frac{\pi}{4}, \frac{\pi}{2}\right)$.
\end{rema}

\subsubsection{Case (1c)}
We only need to find the data for the surface $\circled{9}$ here, so assume that $c = 0$ and $d = -1$.
We align the axial directions so that $\vec{v}_1 \parallel \vec{e}_1 + \vec{e}_0$ and $\vec{v}_2 \parallel \vec{e}_2$, implying that the unit normal vector is
	\[
		N = \left(N_0 \pm \frac{\rho_x}{\rho},\, \pm \frac{1}{\sqrt{-d}}\frac{\rho_y}{\rho},\, N_0\right),
	\]
where $N_0$ can be found from the fact that $N$ has unit length.
After making the parameter shift $x \mapsto x - \log2$, we calculate the Weierstrass data as
	\begin{equation}\label{eqn:wData3}
		h_3(z) =\phoro{e}^z + j, \quad
		\eta_3 = \frac{1}{2}\phoro{e}^{-z}.
	\end{equation}

\subsubsection{Case (1d)}
Here, we have that $d < c \leq 0$.
Align the axial directions so that $\vec{v}_1 \parallel \vec{e}_0$ and $\vec{v}_2 \parallel a_1\vec{e}_1 + a_2 \vec{e}_2$.
In this case, we let $c = -c_4^2$ and $d = -c_4^2 - 1$ for $c_4 \geq 0$, and let $a_1 = 1$ and $a_2 = c_4$  (see Figure \ref{fig:pathS4}).
Then, the unit normal vector is
	\[
		N = \left(-c_4N_2 \pm \frac{\rho_y}{\rho}, N_2,\, \pm \frac{1}{c_4}\frac{\rho_x}{\rho}\right).
	\]
Using this, after a shift of parameter $y \mapsto y - \log(\sqrt{1 + c_4^2})$, we obtain that
	\begin{equation}\label{eqn:wData4}
		h_4^{c_4}(z) = j \phoro{e}^{jz} + c_4, \quad \eta_4^{c_4} = \frac{1}{2} \phoro{e}^{-j z}.
	\end{equation}


\subsubsection{Case (2)}
Finally, we assume that $\Box\rho \equiv 0$, and by Remark \ref{rema:global}, we only consider the case $\rho(x,y) = -y$.
We assume that the axial direction is $\vec{v}_2 = \vec{e}_1 + \vec{e}_0$.
Then similar to Lemma \ref{lemm:normal}, we can calculate that
	\[
		N_1 - N_0 = \frac{\rho_y}{\rho}.
	\]
Since $\rho_x \equiv 0$, we have that $N(x,y)$ has the form $N(x,y) = (N_1(y),\, 0,\, N_0(y))\cdot T(x)$ for an isometry transform $T(x) \in \mathrm{SO(2,1)}$ keeping the lightlike axis $\vec{v}_2$. Hence, we obtain the following lemma.

\begin{lemm}
	If $\rho(x,y) = y$, then the unit normal vector $N$ is given by
	\begin{align*}
		N(x,y) &= \left(\frac{y^2+1}{2y}, 0, \frac{y^2 - 1}{2y}\right) \cdot
			\begin{pmatrix}
				1-\frac{x^2}{2} & x & -\frac{x^2}{2}\\
				-x & 1 & -x\\
				\frac{x^2}{2} & -x & 1+\frac{x^2}{2}
			\end{pmatrix}
			= \left(\frac{1 - x^2 + y^2}{2y},\, \frac{x}{y},\, -\frac{1 + x^2 - y^2}{2y}\right).
	\end{align*}
\end{lemm}

Therefore, we recover the Weierstrass data as follows:
	\begin{equation}\label{eqn:wData5}
		h_5(z) = \frac{z + j}{1 - j z}, \quad
		\eta_5(z) = \frac{1}{4}(j z - 1)^2.
	\end{equation}
Finally, by considering $(h_5, \eta_5 \dif z) \mapsto (h_5, \lambda^{-2}\eta_5 \dif z)$ for $\lambda$ as in Remark \ref{rema:global}, we obtain the Weierstrass data for surfaces $\circled{11}$ and $\circled{12}$.

In summary, we obtain the following complete classification of timelike minimal surfaces with planar curvature lines.

\begin{theo}\label{theo:wData}
	A generalized timelike minimal surface with planar curvature lines in Minkowski $3$-space must be a piece of one, and only one, of
	\begin{itemize}
		\item plane \textup{(P)} $(0, \dif z)$,
		\item[$\circled{\emph{1}}$] timelike catenoid with timelike axis \textup{(C\textsubscript{T})} $\left(\ptan z, \frac{1}{2}\pcos^2 z \dif z\right)$,
		\item[$\circled{\emph{2}}$] doubly periodic timelike minimal Bonnet-type surface (with timelike axial direction) \textup{(B\textsubscript{Tper})}
			\[
				\left\{ \left(\tilde{c}_1 \ptan z, \tfrac{1}{2 \tilde{c}_1}\pcos^2 z \dif z\right) : \tilde{c}_1 > 1 \right\},
			\]
		\item[$\circled{\emph{3}}$] timelike Enneper-type surface \textup{(E)} $\left(\sqrt{2} z, \frac{1}{2\sqrt{2}} \dif z\right)$,
		\item[$\circled{\emph{4}}$] timelike minimal Bonnet-type surface with timelike axial direction of first kind \textup{(B\textsubscript{T1})},
			\[
				\left\{ \left(j \phoro{e}^{j z} - j c_2, \tfrac{1}{2}\phoro{e}^{-j z} \dif z\right) : c_2 > 1 \right\},
			\]
		\item[$\circled{\emph{5}}$] immersed timelike catenoid with spacelike axis \textup{(C\textsubscript{S1})} $\left(\phoro{e}^z, \frac{1}{2}\phoro{e}^{-z} \dif z\right)$,
		\item[$\circled{\emph{6}}$] timelike minimal Bonnet-type surface with lightlike axial direction of first kind \textup{(B\textsubscript{L1})}
			\[
				\left(j \phoro{e}^{j z} - j, \tfrac{1}{2}\phoro{e}^{-j z} \dif z\right),
			\]
		\item[$\circled{\emph{7}}$] timelike minimal Bonnet-type surface with spacelike axial direction \textup{(B\textsubscript{S})},
			\[
				\left\{ \left(j \phoro{e}^{j z} - j c_2, \tfrac{1}{2}\phoro{e}^{-j z} \dif z\right) : 0 < c_2 < 1 \right\},
			\]
		\item[$\circled{\emph{8}}$] non-immersed timelike catenoid with spacelike axis \textup{(C\textsubscript{S2})} $\left(j \phoro{e}^{j z}, \frac{1}{2}\phoro{e}^{-j z} \dif z\right)$,
		\item[$\circled{\emph{9}}$] timelike minimal Bonnet-type surface with lightlike axial direction of second kind \textup{(B\textsubscript{L2})} $(\phoro{e}^z + j, \frac{1}{2} \phoro{e}^{-z} \dif z)$,
		\item[$\circled{\emph{10}}$] timelike minimal Bonnet-type surface with timelike axial direction of second kind \textup{(B\textsubscript{T2})}
			\[
				\left\{ \left( j \phoro{e}^{jz} + c_4,\tfrac{1}{2} \phoro{e}^{-j z} \dif z\right) : c_4 > 0 \right\},
			\]
		\item[$\circled{\emph{11}}$] timelike catenoid with lightlike axis \textup{(C\textsubscript{L})} $\left(\frac{z + j}{1 - j z}, \frac{1}{4}(j z - 1)^2 \dif z\right)$, or one member of its associated family $\circled{\emph{12}}$,
	\end{itemize}
	given with their respective Weierstrass data.
\end{theo}

\begin{figure}
	\centering
	\savebox{\mybox}{\includegraphics[width=0.23\textwidth]{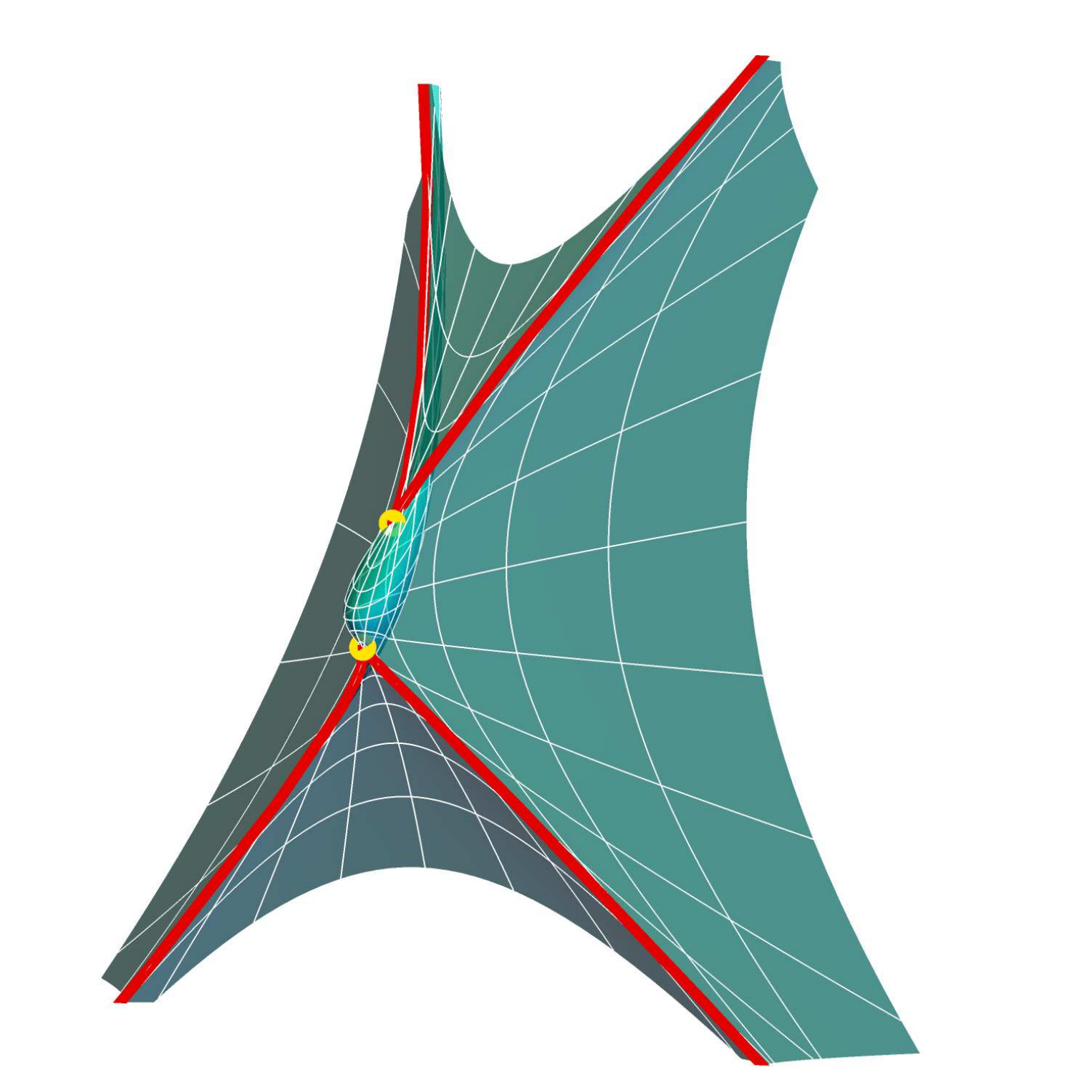}}
	\begin{subfigure}{0.23\textwidth}
		\vbox to \ht\mybox{%
			\vfill
			\includegraphics[width=1\textwidth]{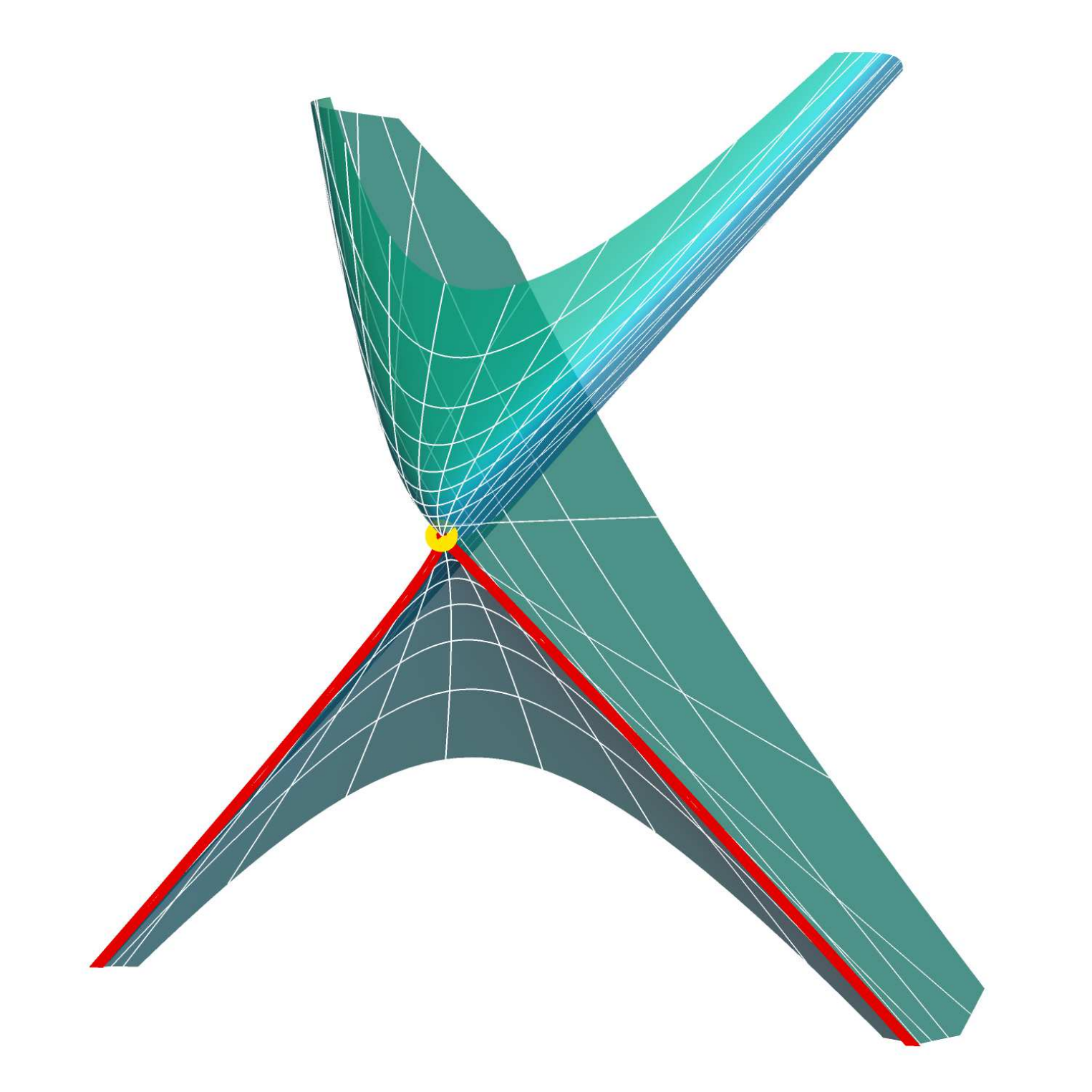}
			\vfill
		}
		\caption{\textup{B\textsubscript{T1}}}
	\end{subfigure}
	~
	\begin{subfigure}{0.23\textwidth}
		\vbox to \ht\mybox{%
			\vfill
			\includegraphics[width=1.0\textwidth]{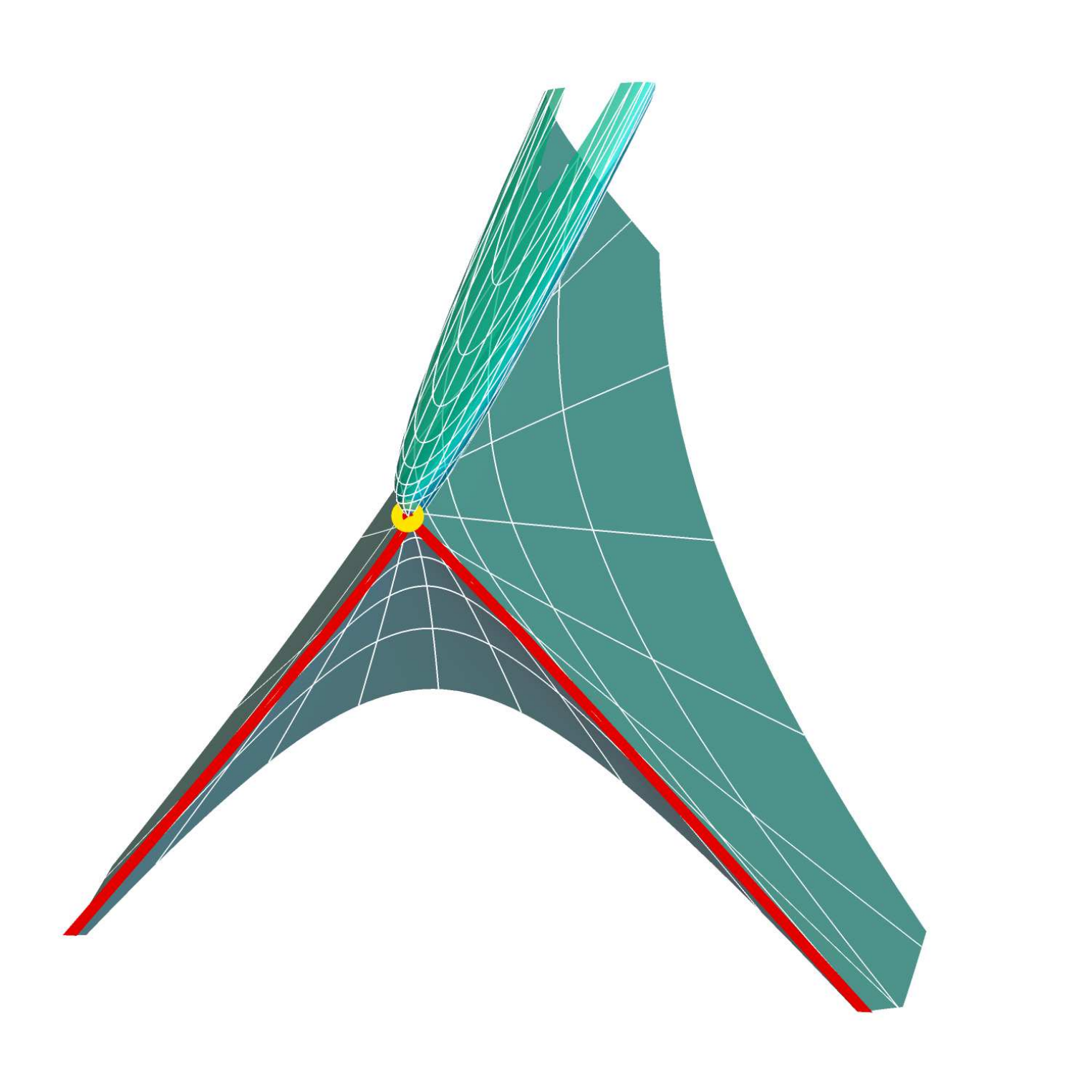}
			\vfill
		}
		\caption{\textup{B\textsubscript{L1}}}
	\end{subfigure}
	~
	\begin{subfigure}{0.23\textwidth}
		\usebox{\mybox}
		\caption{\textup{B\textsubscript{S}}}
	\end{subfigure}
	\par\bigskip
	\savebox{\mybox}{\includegraphics[width=0.23\textwidth]{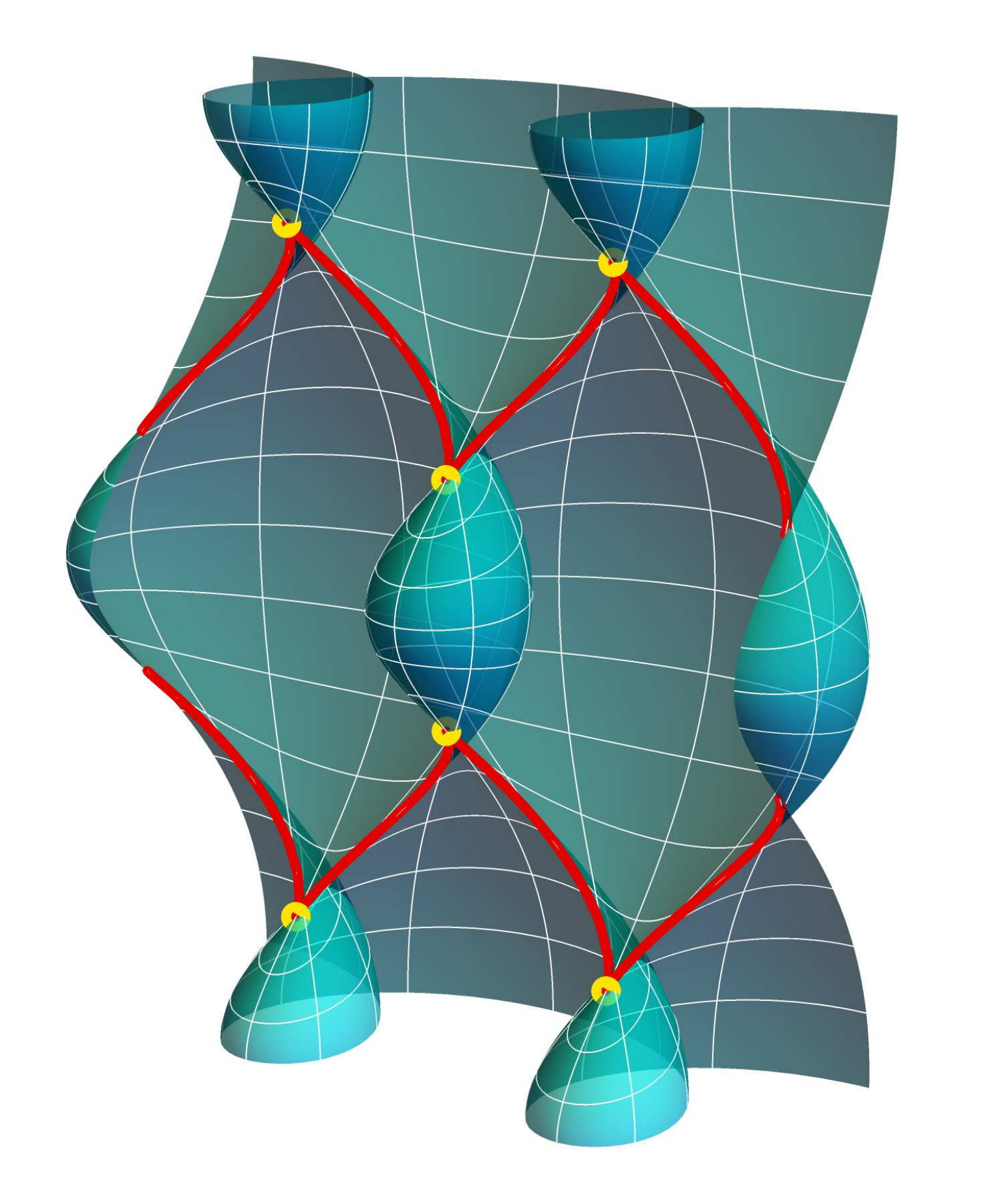}}
	\begin{subfigure}{0.23\textwidth}
		\usebox{\mybox}
		\caption{\textup{B\textsubscript{Tper}}}
	\end{subfigure}
	~
	\begin{subfigure}{0.23\textwidth}
		\vbox to \ht\mybox{%
			\vfill
			\includegraphics[width=1.0\textwidth]{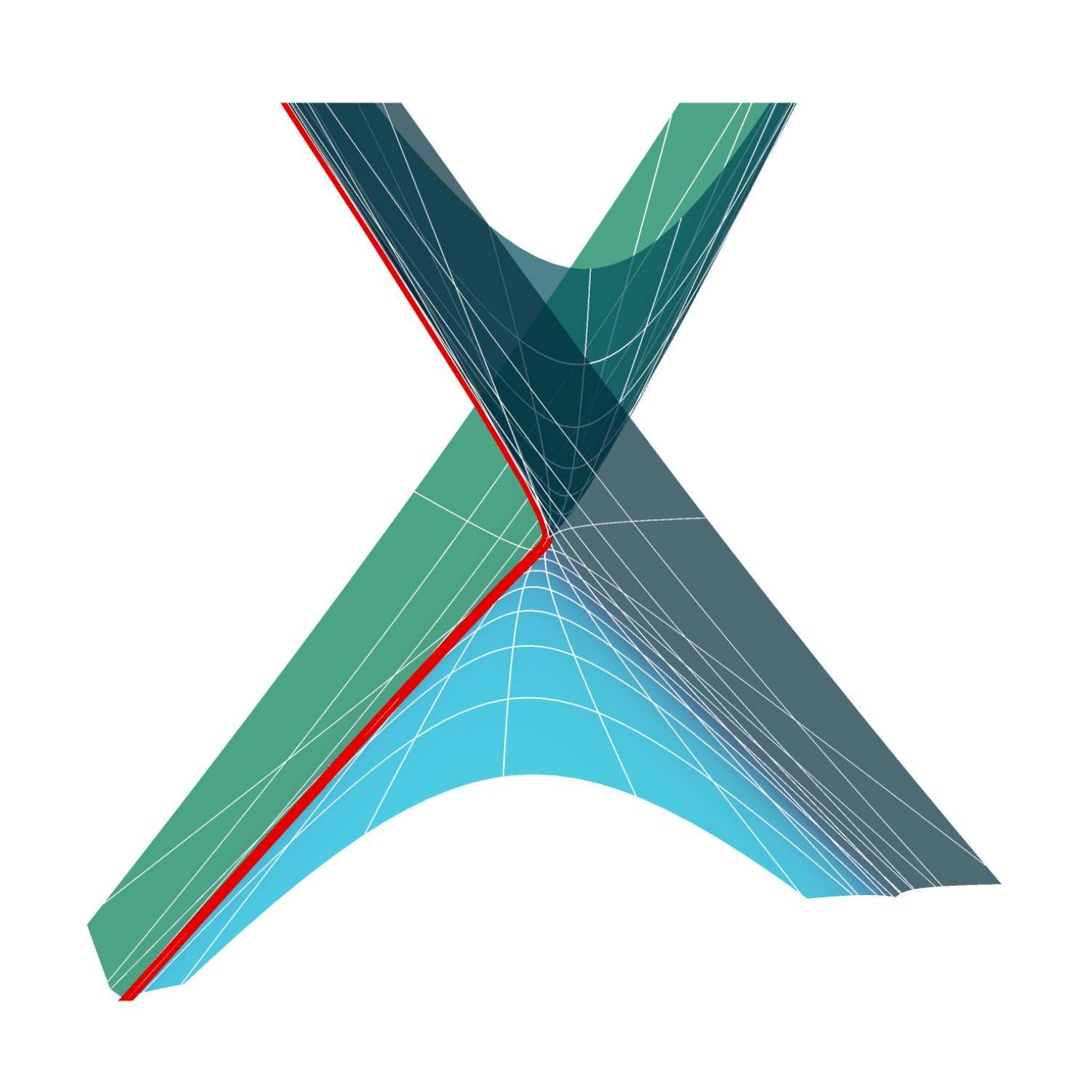}
			\vfill
		}
		\caption{\textup{B\textsubscript{T2}}}
	\end{subfigure}
	~
	\begin{subfigure}{0.23\textwidth}
		\vbox to \ht\mybox{%
			\vfill
			\includegraphics[width=1.0\textwidth]{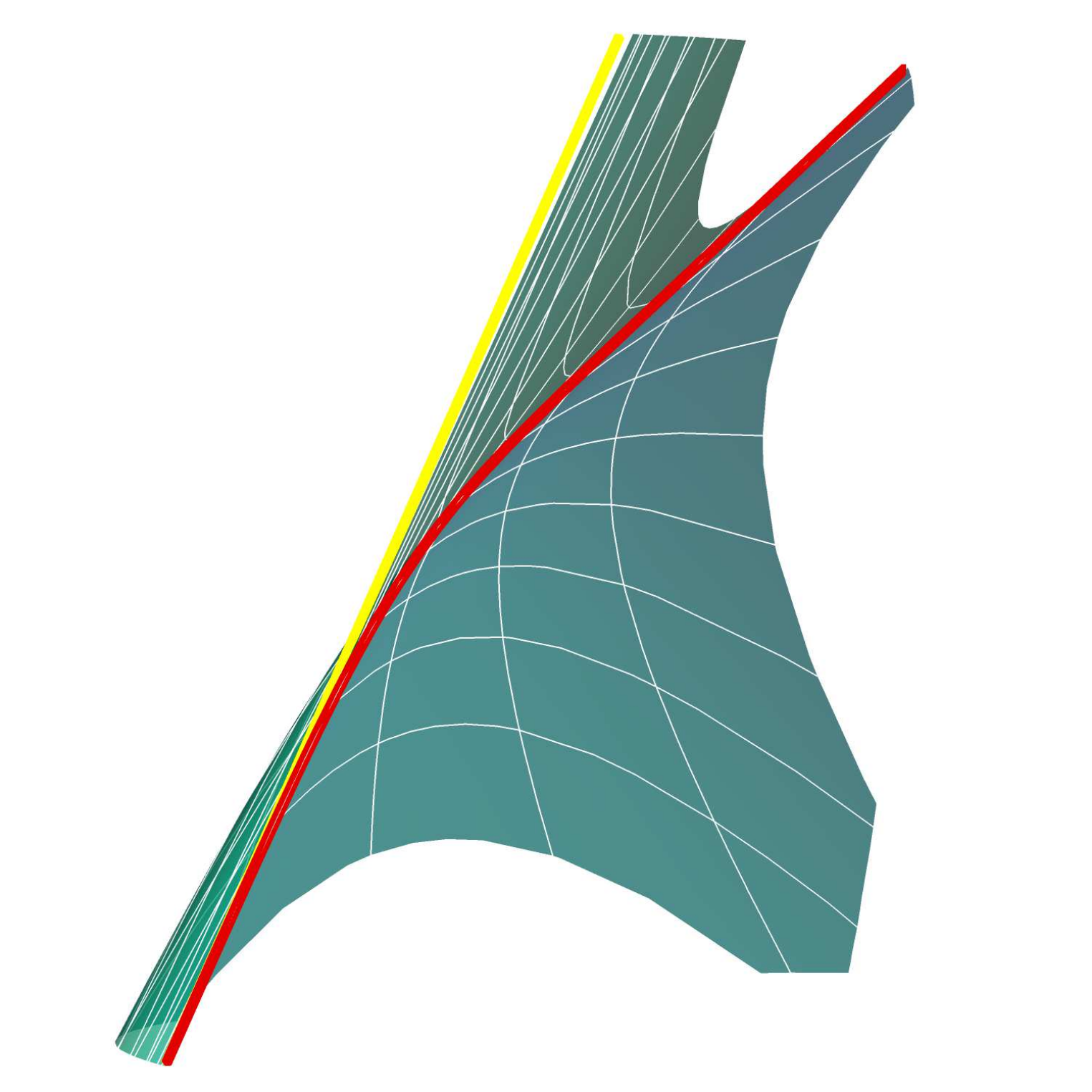}
			\vfill
		}
		\caption{\textup{B\textsubscript{L2}}}
		\label{fig:lightlikeline}
	\end{subfigure}
	\caption{Timelike minimal Bonnet-type surfaces. As in Remark \ref{rema:global}, we treat these as generalized timelike minimal surfaces, admitting singularities, where the singularities on these surfaces are highlighted.}
	\label{fig:singularities}
\end{figure}

\section{Null curves of timelike minimal surfaces with planar curvature lines}\label{Sec.3}
In this section, we consider timelike minimal surfaces with planar curvature lines in terms of their generating null curves (see Fact \ref{fact:Weierstrass2}). First, we introduce the theory of null curves in $\mathbb{R}^{2,1}$. For an in-depth discussion of the theory of null curves, we refer the readers to works such as \cite{vessiot_sur_1905, bonnor_null_1969, ferrandez_null_2001, inoguchi_null_2008, lopez_differential_2014, olszak_note_2015}.

\subsection{Frenet-Serret type formula for non-degenerate null curves}\label{sec1}

A regular curve $\gamma=\gamma(t)\colon I\rightarrow \mathbb{R}^{2,1}$ is called a \emph{null curve} if 
	\[
		\langle \gamma', \gamma' \rangle=0, 
	\]
and $\gamma$ is said to be \emph{non-degenerate} if $\gamma'$ and $\gamma''$ are linearly independent at each point on $I$. (Here, $'$ denotes  $\frac{\dif}{\dif t}$.)
For a non-degenerate null curve $\gamma(t)$, we can normalize (see {\cite[Section $2$]{olszak_note_2015}, for example) the parameter so that 
	\begin{equation}\label{eq:psudo-arc}
		\langle \ddot\gamma(s), \ddot\gamma(s)\rangle=1,
	\end{equation}
where $\dot{\phantom{s}}$ denotes $\frac{\dif}{\dif s}$.
A parameter $s$ satisfying \eqref{eq:psudo-arc} is called the \textit{pseudo-arclength parameter}, introduced in \cite{bonnor_null_1969}.
From now on, let $s$ denote a pseudo-arclength parameter.
If we take the vector fields
	\[
		\vect{\sigma}(s) := \dot\gamma(s),\quad
		\vect{e}(s) := \ddot\gamma(s),
	\]
and then there is a unique null vector field $\vect{n}$ such that
	\[
		\langle \vect{n}, \vect{\sigma} \rangle=-2,\quad
		\langle\vect{n}, \vect{e}\rangle=0.
	\]
If we set the \emph{lightlike curvature} (see \cite[p.47]{inoguchi_null_2008}) of $\gamma$ to be
	\begin{equation}\label{eq:nullcurvature}
		\kappa_\gamma(s) := -\inner{\dot{\vect{n}}(s)}{\vect{e}(s)},
	\end{equation}
we get
	\[
		- \dot{\vect{n}} = \kappa_\gamma \vect{e},\qquad
		\dot{\vect{e}} = -\frac{\kappa_\gamma}{2} \vect{\sigma} + \frac1{2} \vect{n}.
	\]
Therefore, we obtain the following Frenet-Serre type formula for non-degenerate null curves.

\begin{prop}[cf.~\cite{inoguchi_null_2008, lopez_differential_2014}]\label{prop:Frenet}
	For a non-degenerate null curve $\gamma$ parametrized by pseudo-arclength parameter, the null frame $\mathcal{F}:=\{\vect{\sigma}, \vect{e}, \vect{n}\}$ satisfies
		\[
			\mathcal{F}' = \mathcal{F}
				\begin{pmatrix}
					0 & -\kappa_\gamma/2 & 0 \\
					1 & 0 & -\kappa_\gamma \\
					0 & 1/2 & 0
				\end{pmatrix}.
		\]
	Moreover, the lightlike curvature $\kappa_\gamma$ of $\gamma$ is written as
		\begin{equation}\label{eq:pseudo_arc}
			\kappa_\gamma =\langle \dddot\gamma, \dddot\gamma\rangle.
		\end{equation}
\end{prop}

\begin{example}\label{ex:const_curvature}
A non-degenerate null curve parametrized by pseudo-arclength with constant lightlike curvature $\kappa_\gamma$ is called a \emph{null helix} in \cite{ferrandez_null_2001,inoguchi_null_2008}, and such curves have been studied by many authors.
Any null helix $\gamma$ is congruent to one of the following:
\[
	\begin{cases}
		\gamma(s) = \frac{1}{\kappa_\gamma}\left(\cos{(cs)}, \sin{(cs)}, cs \right), &\text{ when $\kappa_\gamma=c^2 > 0$},\\
		\gamma(s) = \left(\frac{s^2}{2},-\frac{s^3}{6}+\frac{s}{2}, \frac{s^3}{6}+\frac{s}{2} \right), &\text{ when $\kappa_\gamma= 0$}, \\
		\gamma(s) = \frac{1}{\kappa_\gamma}\left(cs, \cosh{(cs)}, \sinh{(cs)} \right), &\text{ when $\kappa_\gamma=-c^2 < 0$}.
	\end{cases}
\]
\end{example}

\subsection{Characterization of timelike minimal surfaces with planar curvature lines}

Using the theory of null curves, we now characterize timelike minimal surfaces with planar curvature lines in terms of its generating null curves.
We first remark on the relationship between the generating null curves and the normalization of the Hopf differential factor.

\begin{lemm}[cf.\ p.\ 347 of \cite{inoguchi_timelike_2004}]\label{lemma:normalization}
The normalization of the Hopf differential factor $\phoro{q}=-\frac{1}{2}$ of a timelike minimal surface $F$ implies that the generating null curves are parametrized by pseudo-arclength.
\end{lemm}
\begin{proof}
Let $F$ be represented via two generating null curves $\alpha(u)$ and $\beta(v)$ as in \eqref{eq:null_decomp}.
By the Gauss-Weingarten equations, we have
	\[
		\alpha_{uu}(=2F_{uu})=2\frac{\rho_u}{\rho}\alpha_u-N,\quad \beta_{vv}(=2F_{vv})=2\frac{\rho_v}{\rho}\beta_v-N,
	\]
where $\rho$ is the Lorentz conformal factor of the first fundamental form and $N$ is the unit normal of $F$.
Therefore, we can check that
	\[
		\langle \alpha_{uu},\alpha_{uu} \rangle =\langle \beta_{vv},\beta_{vv} \rangle=\langle N,N \rangle=1,
	\]
i.e.\ $u$ and $v$ are pseudo-arclength parameters of $\alpha(u)$ and $\beta(v)$, respectively.
\end{proof}

Now we state and prove the theorem relating the lightlike curvatures of the generating null curves and the planar curvature line condition \eqref{eqn:pde2}.
\begin{theo}\label{thm:curvature_characterization2}
Away from flat points and singular points, a timelike minimal surface $F$ has planar curvature lines if and only if it has negative Gaussian curvature, and its generating null curves have the same constant lightlike curvature. 
\end{theo}
\begin{proof}
We consider a timelike minimal surface $F$ written as in \eqref{eq:null_decomp}, with its first fundamental form as in \eqref{eqn:firstFundamental}.
For $z = x + j y = (u+v)/2 + j(u-v)/2$, the planar curvature line condition \eqref{eqn:pde2} can be expressed as
\begin{equation} \label{eq:PCLcondition}
	\rho_{uu} - \rho_{vv}=0.
\end{equation}

Let us take the null frames $\mathcal{F}_\alpha:=\{\vect{\sigma}_\alpha, \vect{e}_\alpha, \vect{n}_\alpha\}$ for $\alpha$ and $\mathcal{F}_\beta:=\{\vect{\sigma}_\beta, \vect{e}_\beta, \vect{n}_\beta\}$ for $\beta$ as in Section \ref{sec1}.
By using the frame of the surface $F$, we can check that $\vect{n}_\alpha$ and $\vect{n}_\beta$ are written as
	\begin{equation}\label{eq:n'}
		\vect{n}_\alpha=\left(\frac{2\rho_u}{\rho}\right)^2\alpha_u-\frac{1}{\rho^2}\beta_v-4\frac{\rho_u}{\rho}N,\quad
		\vect{n}_\beta=-\frac{1}{\rho^2}\alpha_u +\left(\frac{2\rho_v}{\rho}\right)^2\beta_v-4\frac{\rho_v}{\rho}N.
	\end{equation}
Since we normalized the Hopf differential factor as $\phoro{q}=-\frac{1}{2}$, we have that $u$ and $v$ are pseudo-arclength parameters of $\alpha$ and $\beta$ by Lemma \ref{lemma:normalization}; hence, we can take
	\[
		\vect{e}_\alpha=\alpha_{uu} \quad\text{and}\quad \vect{e}_\beta=\beta_{vv}.
	\]
By the Gauss-Weingarten equations, ${(\vect{e}_\alpha)}_u$ and ${(\vect{e}_\beta)}_v$ can be expressed as
\begin{equation}\label{eq:e'}
	\begin{aligned}
		{(\vect{e}_\alpha)}_u &= 2\frac{\rho_{uu}\rho+{\rho_u}^2}{\rho^2}\alpha_u-\frac{1}{2\rho^2}\beta_v-2\frac{\rho_u}{\rho}N \\
		{(\vect{e}_\beta)}_v &= -\frac{1}{2\rho^2}\alpha_u +2\frac{\rho_{vv}\rho+{\rho_v}^2}{\rho^2}\beta_v-2\frac{\rho_v}{\rho}N.
	\end{aligned}
\end{equation}
By \eqref{eq:nullcurvature}, \eqref{eq:n'} and \eqref{eq:e'}, the lightlike curvatures $\kappa_\alpha$ and $\kappa_\beta$ of the generating null curves $\alpha$ and $\beta$ can be calculated as
	\begin{equation}\label{eq:null_curvatures}
		\kappa_\alpha=\langle \vect{n}_\alpha, (\vect{e}_\alpha)_u\rangle=-4\frac{\rho_{uu}}{\rho},\quad
		\kappa_\beta=\langle \vect{n}_\beta, (\vect{e}_\beta)_v\rangle=-4\frac{\rho_{vv}}{\rho},
	\end{equation}
and hence 
	\[
		\kappa_\alpha(u)-\kappa_\beta(v)=-4\frac{\rho_{uu}-\rho_{vv}}{\rho}.
	\]
Therefore, we conclude that $\kappa_\alpha$ are $\kappa_\beta$ are the same constant if and only if the Lorentz conformal factor $\rho$ satisfies the planar curvature line condition \eqref{eq:PCLcondition}.
\end{proof}


\subsection{Deformations of null curves with constant lightlike curvature}\label{sect:deformationnull}
We now consider continuous deformations of null curves preserving their pseudo-arclength parametrization and constancy of lightlike curvatures.
As in \cite{cho_deformation_2017, cho_deformation_2018}, we consider a deformation to be ``continuous'' with respect to a parameter if the deformation dependent on the parameter converges uniformly over compact subdomains component by component.
First, we introduce how the lightlike curvatures of generating null curves are determined for a timelike minimal surface with planar curvature lines.

\begin{prop}
	The lightlike curvatures $\kappa_\alpha$ and $\kappa_\beta$ of the generating null curves of a timelike minimal  surface with planar curvature lines are given by the constants $c$ and $d$ in \eqref{eqn:ode} via
	\begin{equation}\label{eq:c,d,kappa}
		\kappa_\alpha=\kappa_\beta = d-c.
	\end{equation}
\end{prop}
\begin{proof}
By \eqref{eq:null_curvatures}, we have
	\[
		\kappa_\alpha=\kappa_\beta =-2\frac{\rho_{uu}+\rho_{vv}}{\rho} =-\frac{\rho_{xx}+\rho_{yy}}{\rho}.
	\]
Using \eqref{eqn:ode1} and \eqref{eqn:ode3}, we prove the desired relation.
\end{proof}

A deformation of a timelike minimal surface with planar curvature lines corresponds to a deformation of its generating null curves, which have constant lightlike curvatures.
Therefore, by using the relation \eqref{eq:c,d,kappa}, we can deform null curves with constant curvature preserving the pseudo-arclength parametrization and the constancy of lightlike curvature (each of the null curves may have different constant lightlike curvature). 

As an example, we give a deformation of null curves coming from the surfaces in case (1a).
To do this, we first consider a slightly modified method of the one we used to obtain the Weierstrass data \eqref{eqn:wData1} in Section \ref{sssec:sheet1}.
Since $c \geq 0$ and $d \geq 0$, we define $\gamma$ and $\delta$ so that $\gamma^2 = c$ and $\delta^2 = d$.
Let $\gamma = 2^{1/4} \cos{\hat{c}_1}$ and $\delta = 2^{1/4} \sin{\hat{c}_1}$ for $\hat{c}_1 \in \left(-\frac{\pi}{4}, \frac{3\pi}{4}\right)$ (see Figure \ref{fig:pathD1}).
\begin{figure}
	\includegraphics[width=0.3\textwidth]{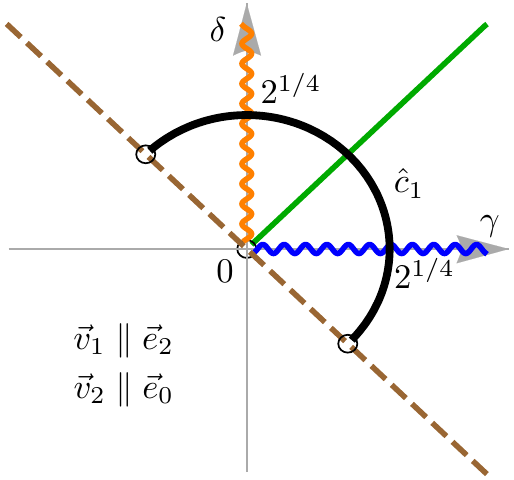}
	\caption{Modified version of a path in case (1a) to include the timelike plane in the deformation. For the meaning of the diagram, see Figure \ref{fig:bifurcation}.}
	\label{fig:pathD1}
\end{figure}

Then we can calculate similarly as before to obtain that
	\begin{equation}\label{eqn:wDataDeform1}
		\begin{aligned}
			h_{\textup{P}}^{\hat{c}_1}(z) &= \begin{cases}
					\mathrlap{\frac{\sqrt{\cos{(2 \hat{c}_1)}}}{\cos \hat{c}_1 - \sin \hat{c}_1} \ptanh\left(\frac{\sqrt{\cos{(2 \hat{c}_1)}}}{2^{3/4}} \,z\right),}\hphantom{\frac{1}{2^{1/4}(\cos{\hat{c}_1} + \sin{\hat{c}_1)}} \pcos^2 \left(\frac{\sqrt{-\cos{(2 \hat{c}_1)}}}{2^{3/4}} \,z\right),} &\text{if $\hat{c}_1 \in (-\frac{\pi}{4},\frac{\pi}{4})$},\\
					\frac{1}{2^{1/4}}z, & \text{if $\hat{c}_1 = \frac{\pi}{4}$},\\
					-\frac{\sqrt{-\cos{(2 \hat{c}_1)}}}{\cos \hat{c}_1 - \sin \hat{c}_1} \ptan\left(\frac{\sqrt{-\cos{(2 \hat{c}_1)}}}{2^{3/4}} \,z\right), & \text{if $\hat{c}_1 \in (\frac{\pi}{4},\frac{3\pi}{4})$},
				\end{cases}\\
			\eta_{\textup{P}}^{\hat{c}_1}(z) &= \begin{cases}
					\frac{1}{2^{1/4}(\cos{\hat{c}_1} + \sin{\hat{c}_1)}} \pcosh^2 \left(\frac{\sqrt{\cos{(2 \hat{c}_1)}}}{2^{3/4}} \,z\right), &\text{if $\hat{c}_1 \in (-\frac{\pi}{4},\frac{\pi}{4})$},\\
					\frac{1}{2^{3/4}}, & \text{if $\hat{c}_1 = \frac{\pi}{4}$},\\
					\frac{1}{2^{1/4}(\cos{\hat{c}_1} + \sin{\hat{c}_1)}} \pcos^2 \left(\frac{\sqrt{-\cos{(2 \hat{c}_1)}}}{2^{3/4}} \,z\right), &\text{if $\hat{c}_1 \in (\frac{\pi}{4},\frac{3\pi}{4})$}.
				\end{cases}
		\end{aligned}
	\end{equation}
\begin{rema}
By noticing that $\gamma^2 = 2^{-3/2}(4\cos^2{\hat{c}_1})$ and $\delta^2 = 2^{-3/2}(4\sin^2{\hat{c}_1})$, and the fact that a homothety in the $(c,d)$-plane amounts to a homothety in the $(x,y)$-plane, one can also get the parameromorphic data $h_{\textup{P}}^{\hat{c}_1}(z)$ of \eqref{eqn:wDataDeform1} from that of \eqref{eqn:wData1} by applying a homothety change in the domain $z \mapsto 2^{-3/4} z$.
\end{rema}

Now to get the parametrization, let $F_{\textup{P}}^{\hat{c}_1} (x,y)$ be defined from $(h_{\textup{P}}^{\hat{c}_1}, \eta_{\textup{P}}^{\hat{c}_1} \dif z)$ via the Weierstrass-type representation in Fact \ref{fact:Weierstrass1}.
We define
	\begin{equation}\label{eqn:sheet1Deformation}
		\hat{F}_{\textup{P}}^{\hat{c}_1} (x,y) = R^{\hat{c}_1}\left(F_{\textup{P}}^{\hat{c}_1} (x,y) - F_{\textup{P}}^{\hat{c}_1} (0,0)\right),
	\end{equation}
where
	\begin{equation}\label{eqn:homothety}
		R^{\hat{c}_1} = \left(1- \sin{\left(\hat{c}_1 + \tfrac{\pi}{4}\right)}\right)\left| \cos{2 \hat{c}_1} \right| +  \sin{\left(\hat{c}_1 + \tfrac{\pi}{4}\right)}.
	\end{equation}
A straightforward calculation then shows that
	\begin{gather*}
		\lim_{\hat{c}_1 \to \tfrac{\pi}{4}} \hat{F}_{\textup{P}}^{\hat{c}_1} (2^{1/4}x, 2^{1/4}y) = \tfrac{1}{\sqrt{2}}\left( x^2 + y^2, \, x - x y^2 - \tfrac{1}{3} x^3 , \, - y - x^2 y - \tfrac{1}{3} y^3\right),\\
		\lim_{\hat{c}_1 \searrow -\tfrac{\pi}{4}} \hat{F}_{\textup{P}}^{\hat{c}_1} (x,y) = \left( 0, \, \tfrac{3}{2^{3/4}} x, -\tfrac{3}{2^{3/4}} y \right) = \lim_{\hat{c}_1 \nearrow \tfrac{3\pi}{4}} \hat{F}_{\textup{P}}^{\hat{c}_1} (x,y),
	\end{gather*}
implying that $\hat{F}_{\textup{P}}^{\hat{c}_1} (x,y)$ for $\hat{c}_1 \in \left[-\frac{\pi}{4},\frac{3\pi}{4}\right]$ gives a continuous deformation consisting of every surface in case (1a), including the timelike minimal Enneper-type surface and the timelike plane.

To obtain a deformation of null curves from the surface, let us now take $A_1=\sqrt{\cos{2\hat{c}_1}}$.
After applying a suitable homothety to the domain, the generating null curves of the surfaces in the case (1a) discussed in \eqref{eqn:sheet1Deformation} are written as
\begin{align*}
	\alpha^{\hat{c}_1}(s) &= \left(\tfrac{\sinh^2{(A_1s)}}{2A_1^2},\tfrac{2 A_1 s \cos{\hat{c}_1} - \sin{\hat{c}_1}\sinh{(2A_1s)}}{4A_1^3},\tfrac{2 A_1 s \sin{\hat{c}_1} - \cos{\hat{c}_1}\sinh{(2A_1s)}}{4A_1^3}\right),\\
	\beta^{\hat{c}_1}(s) &= \alpha^{\hat{c}_1}(s)\cdot
			\left(\begin{smallmatrix}
				1 & 0 & 0\\
				0 & 1 & 0\\
				0 & 0 & -1
			\end{smallmatrix}\right),
\end{align*}
i.e.\
	\[
		\tfrac{1}{2}(\alpha^{\hat{c}_1}(u) + \beta^{\hat{c}_1}(v)) = \tfrac{1}{2^{3/2} R^{\hat{c}_1}}\hat{F}_{\textup{P}}^{\hat{c}_1} \left(\tfrac{2^{3/4}}{2}(u + v), \tfrac{2^{3/4}}{2}(u - v)\right).
	\]
Note that although $A_1$ is zero at $\hat{c}_1=\frac{\pi}{4}$ and may have complex values, $\alpha^{\hat{c}_1}$ are well-defined non-degenerate null curves for all $\hat{c}_1 \in \left(-\frac{\pi}{4},\frac{3\pi}{4}\right)$.
By Lemma \ref{lemma:normalization}, $s$ is a pseudo-arclength parameter for each $\alpha^{\hat{c}_1}$, and the curves have constant curvature $-4\cos{(2\hat{c}_1)}$.
Moreover, if we apply the scaling factor of the ambient space $R^{\hat{c}_1}$, then we can deform $\alpha^{\hat{c}_1}$ to a lightlike line by considering the directional limit as $\hat{c}_1$ tends to $-\frac{\pi}{4}$ or $\frac{3\pi}{4}$, see Figure \ref{fig:path1_null}.

\begin{figure}
	\centering
	\begin{minipage}{0.495\textwidth}
		\includegraphics[width=\textwidth]{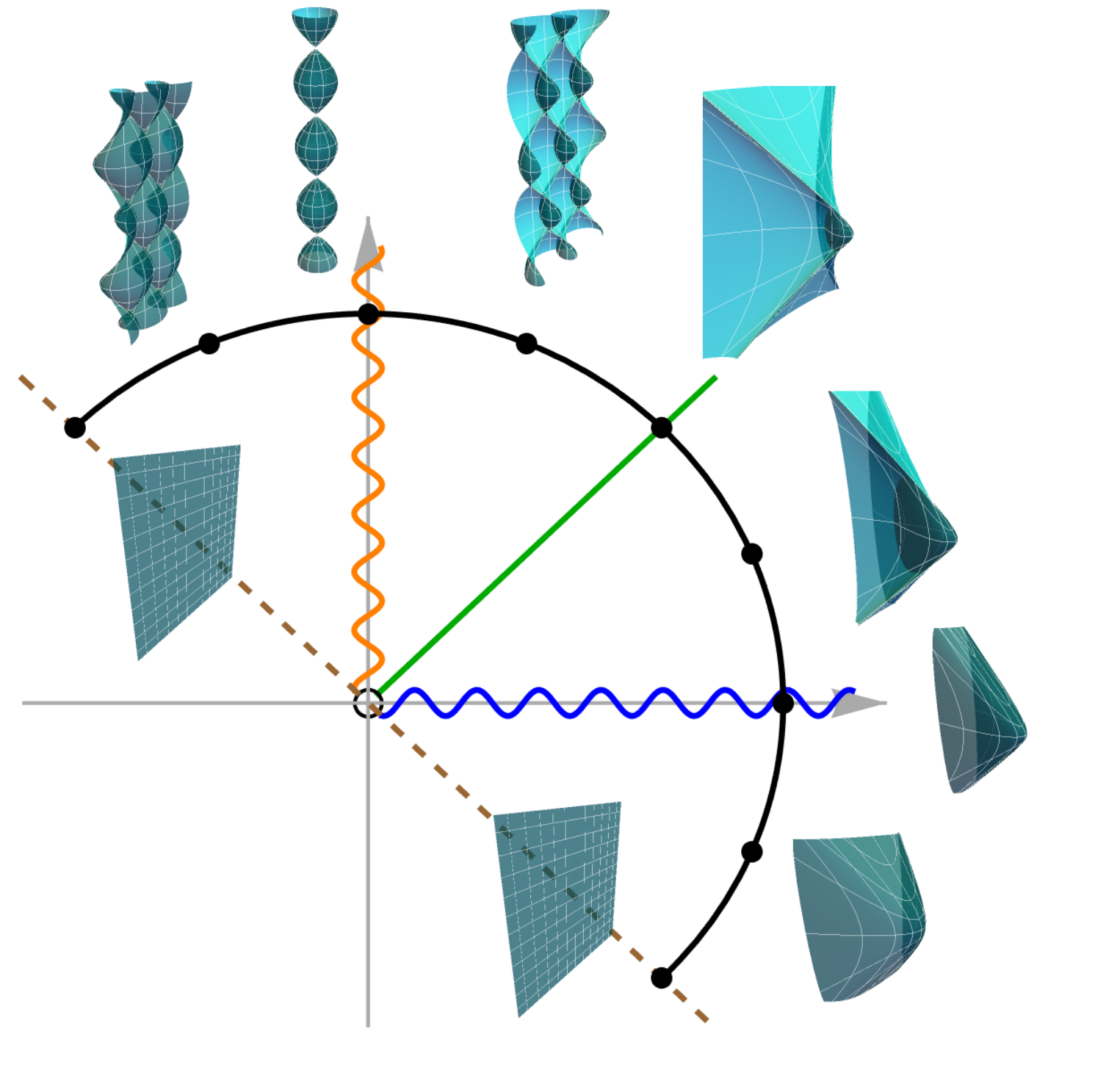}
	\end{minipage}
	\begin{minipage}{0.495\textwidth}
		\includegraphics[width=\textwidth]{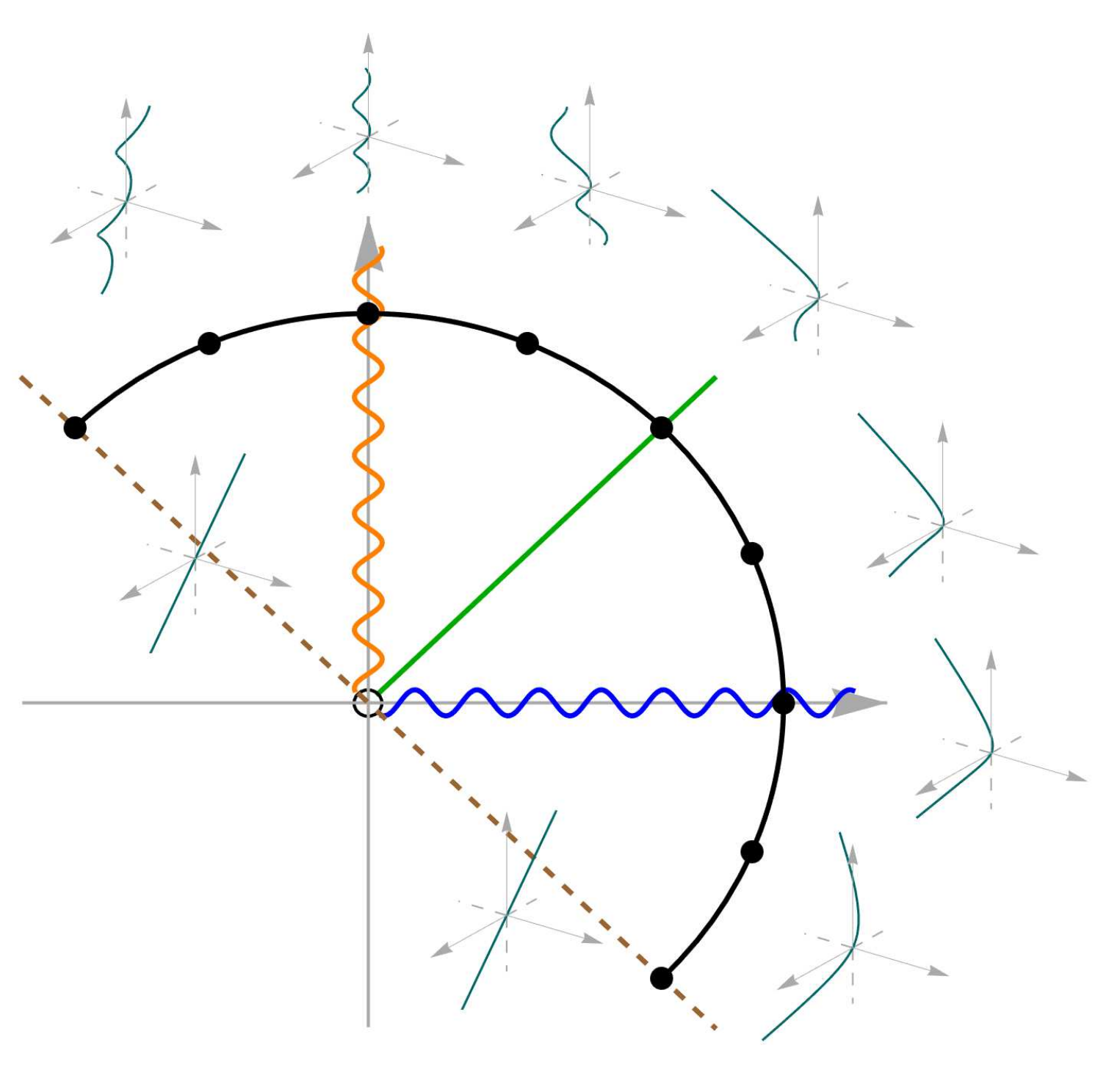}
	\end{minipage}
	\caption{Deformation of null curves with constant lightlike curvature, with their respective surfaces. For the meaning of the diagram, see Figure \ref{fig:bifurcation}.}
	\label{fig:path1_null}
\end{figure}


\section{Geometric characterization of timelike Thomsen surfaces}
Thomsen showed in \cite{thomsen_uber_1923} that the two classes minimal surfaces, those with planar curvature lines, and those that are also affine minimal, called \emph{Thomsen surfaces}, have a striking relationship; namely, they are conjugate minimal surfaces of each other.
Manhart showed in \cite{manhart_bonnet-thomsen_2015} that the analogous result holds for maximal surfaces in $\mathbb{R}^{2,1}$.
In this section, we investigate the relationship between the two classes of timelike minimal surfaces, those with planar curvature lines and those that are also affine minimal.

\subsection{The affine minimal condition -- revisited}\label{sec:4}
A timelike minimal surface which is also affine minimal is called a \emph{timelike Thomsen surface}, defined by Magid in \cite{magid_timelike_1991-1}, who proved the following by applying a result by Manhart \cite{manhart_affinminimalruckungsflachen_1985}. 
	\begin{fact}[\cite{magid_timelike_1991-1}, cf. \cite{manhart_affinminimalruckungsflachen_1985}]\label{thm:Magid}
		Away from flat points, a timelike minimal surface $F$ is affine minimal if and only if on the null coordinates $(u,v)$, there exist functions $\theta=\theta(u)$ and $\vartheta=\vartheta(v)$ such that
		\[
			F_u=(\cos{\theta},\sin{\theta}, 1),\quad F_v=(\cos{\vartheta},\sin{\vartheta}, 1),
		\]
		and ${\dif\theta}/{\dif u}$, ${\dif\vartheta}/{\dif v}$ are both solutions to the equation
		\begin{equation}\label{eq:affine_minimal}
			2\omega^4+2\omega \omega''-\frac{7}{2}\omega'^2-k\omega^3=0\quad \text{for some fixed $k\in \mathbb{R}$},
		\end{equation}
		where $'$ now denotes $\frac{\dif}{\dif u}$ or $\frac{\dif}{\dif v}$.
	\end{fact}
\noindent Magid also solved the above equation explicitly.
\begin{rema}
Milnor \cite{milnor_entire_1990} called the ``angle'' functions $\theta$ and $\vartheta$ the \emph{Weierstrass functions}, and determined the sign of the Gaussian curvature of timelike minimal surfaces using the functions. 
\end{rema} 


In this subsection, we give a geometric interpretation of Fact \ref{thm:Magid} by using the notion of lightlike curvature of non-degenerate null curves. 
Let $\alpha(u)$ and $\beta(v)$ be the generating null curves of a timelike minimal surface $F$ where
	\[
		\alpha(u) = \int^u_{u_0}\left(\cos{\theta}(\tau),\sin{\theta}(\tau), 1\right) \dif\tau + \alpha(u_0),\quad
		\beta(v) = \int^v_{v_0}\left(\cos{\vartheta}(\tau),\sin{\vartheta}(\tau), 1\right) \dif\tau + \beta(v_0)
	\]
for some real constants $u_0$ and $v_0$. Here, we remark that the parameters $u$ and $v$ are not pseudo-arclength parameters. 

In the next proposition, we show that the constant $k$ in the affine minimal equation \eqref{eq:affine_minimal} represents the lightlike curvature of generating null curves, giving a geometric characterization of timelike Thomsen surfaces.
\begin{prop}\label{thm:curvature_characterization}
A timelike minimal surface $F$ satisfies the affine minimal equation \eqref{eq:affine_minimal} if and only if the generating null curves $\alpha$ and $\beta$ of $F$ have the same constant lightlike curvature.
\end{prop}

\begin{proof}
We show that the generating null curves $\alpha$ and $\beta$ must have lightlike curvature $k$.
By the similarity of the argument, it is enough to consider the claim for $\alpha(u)$.

Since $\langle \alpha',\alpha'\rangle=\theta'^2$, we may assume that $\theta'>0$, and we can take the pseudo-arclength
\[
s=\int^u_{u_0}\left(\theta'(\tau) \right)^{1/4} \dif\tau.
\]
By \eqref{eq:pseudo_arc}, we obtain
\begin{equation}\label{eq:curvature_alpha}
	\kappa_\alpha(s) = \dot{u}^6 \left\langle \alpha''', \alpha'''\right\rangle + 9 \dot{u}^2 \ddot{u}^2 \langle \alpha'', \alpha'' \rangle
			+ 6\dot{u}^4\ddot{u} \langle \alpha''', \alpha'' \rangle +2\dot{u}^3\dddot{u} \langle \alpha''', \alpha' \rangle.
\end{equation}
After straightforward calculations, we get
\begin{gather*}
	\langle \alpha'', \alpha'' \rangle = \theta'^2,\quad
		\langle \alpha''', \alpha'' \rangle = \theta' \theta'', \quad
	\langle \alpha''', \alpha' \rangle=-\theta'^2,\quad
		\langle \alpha''', \alpha''' \rangle=\theta''^2+\theta'^4, \\
	\dot{u} = \left(\theta'\right)^{-1/2},\quad
		\ddot{u} = -\frac{\theta''}{2\theta'^2},\quad
		\dddot{u} = \frac{2\theta''^2-\theta'\theta'''}{2\theta'^{7/2}}.
\end{gather*}
Substituting these to \eqref{eq:curvature_alpha}, we obtain
\[
2\theta'^3\kappa_\alpha=2\theta'^4-\frac{7}{2}\theta''^2+2\theta'\theta'''.
\]
Hence, the lightlike curvature $\kappa_\alpha$ is constant if and only if $\omega=\theta'$ satisfies the affine minimal equation \eqref{eq:affine_minimal}.
\end{proof}
In conjunction with the non-degenerate null curves with constant lightlike curvature in Example \ref{ex:const_curvature}, Proposition \ref{thm:curvature_characterization} gives another proof of the classification result of timelike Thomsen surface given in \cite{magid_timelike_1991-1}.
Furthermore, Theorem \ref{thm:curvature_characterization2} and Proposition \ref{thm:curvature_characterization} give us the next theorem relating the two classes of timelike minimal surfaces, a result different from the cases of minimal surfaces in $\mathbb{R}^3$ and maximal surfaces in $\mathbb{R}^{2,1}$.

\begin{theo}\label{thm:Thomsen_Bonnet_relation}
	Let $T$ denote the set of timelike Thomsen surfaces, $B$ the set of timelike minimal surfaces with planar curvature lines, and $B^*$ the conjugates of surfaces in $B$.
	Then,
	\begin{equation}\label{eq:set_decomposition}
		T=B\cup B^*,\quad B\cap B^*=\{\text{timelike planes} \}.
	\end{equation}
\end{theo}


\begin{rema}\label{rema:Thomsen_Bonnet_relation}
Note that for the minimal surface case, the relation between minimal surfaces with planar curvature lines and Thomsen surfaces can be expressed using analogous notations $\tilde{T}$, $\tilde{B}$ and $\tilde{B}^*$, denoting the set of Thomsen surfaces, the set of minimal surfaces with planar curvature lines, and the conjugates of surfaces in $\tilde{B}$, respectively, as:
	\[
		\tilde{T}=\tilde{B}^*,\quad \tilde{B}\cap \tilde{B}^*=\{\text{planes, Enneper surface} \}.
	\]
Similarly, by letting $\hat{T}$, $\hat{B}$ and $\hat{B}^*$ denote the analogous sets for maximal surfaces, respectively, we have that
	\[
		\hat{T}=\hat{B}^*,\quad \hat{B}\cap \hat{B}^*=\left\{
			\begin{gathered}
				\text{spacelike planes, maximal Enneper-type surface,}\\
				\text{associated family of spacelike catenoid with lightlike axis}
			 \end{gathered} \right\}.
	\]
\end{rema}



\subsection{Characterization of the associated family of timelike Thomsen surfaces}
Finally, as a corollary of Theorem \ref{thm:curvature_characterization2} and Proposition \ref{thm:curvature_characterization}, we can also characterize timelike minimal surfaces whose generating null curves have different constant lightlike curvature with the same sign.
\begin{coro}
	Away from flat points, a timelike minimal surface $\tilde{F}$ whose generating null curves $\alpha$ and $\beta$ have constant lightlike curvatures $\kappa_\alpha$ and $\kappa_\beta$ with the same sign is contained in the associated family of a timelike Thomsen surface $F$. In particular, $F$ is either
	\begin{itemize}
		\item a timelike minimal surface with planar curvature lines if $K < 0$, or
		\item the conjugate of a timelike minimal surface with planar curvature lines if $K > 0$. 
	\end{itemize}
	Moreover, such a timelike Thomsen surface $F$ is unique if neither lightlike curvatures of null curves is zero.
\end{coro}
\begin{proof}
As in Remark \ref{rema:asso_family}, the generating null curves of $F^\mu$ are $\alpha^\mu=\mu\alpha$ and $\beta^\mu=\beta/\mu$, with lightlike curvatures $\kappa_{\mu \alpha}={\kappa_\alpha}/\mu$ and $\kappa_{\beta/\mu}=\mu \kappa_{\beta}$, respectively.
Hence, we can take the unique solution $\mu=\sqrt{{\kappa_\alpha}/{\kappa_\beta}}$ to the equation
	\[
		\kappa_{\mu\alpha}=\kappa_{\beta/\mu},\quad \mu>0,
	\]
for which $F^\mu$ is a timelike Thomsen surface.
The surface $F^\mu$ is either in $B$ or $B^*$ depending on the sign of the Gaussian curvature $K$.
\end{proof}

\begin{rema}
	One can also consider the geometric characterization of timelike minimal surfaces whose generating null curves have constant curvatures with different signs.
	By \eqref{eq:null_curvatures}, such a surface can be constructed via the equation
	\[
		\kappa_\alpha + \kappa_\beta=-4\frac{\rho_{uu}+\rho_{vv}}{\rho}=0.
	\]
	We do know that such surface is not in the set $T$ as in \eqref{eq:set_decomposition}.
	However, the geometric qualities of such surfaces are unknown.
\end{rema}

\appendix
\section{Deformation of timelike Thomsen surfaces}\label{sect:deformation}

In this section, we show that there exists a continuous deformation consisting exactly of all timelike Thomsen surfaces.
We do this by first showing that there exists a continuous deformation consisting exactly of all timelike minimal surfaces with planar curvature lines, and then applying the result that relates these surfaces to timelike Thomsen surfaces.

We have already shown in Section \ref{sect:deformationnull} that every surface in case (1a), including the timelike minimal Enneper-type surface, and the timelike plane are conjoined by a continuous deformation given by $\hat{F}_{\textup{P}}^{\hat{c}_1} (x,y)$ in \eqref{eqn:sheet1Deformation}.

\begin{figure}
	\centering
	\savebox{\mybox}{\includegraphics[width=0.2\textwidth]{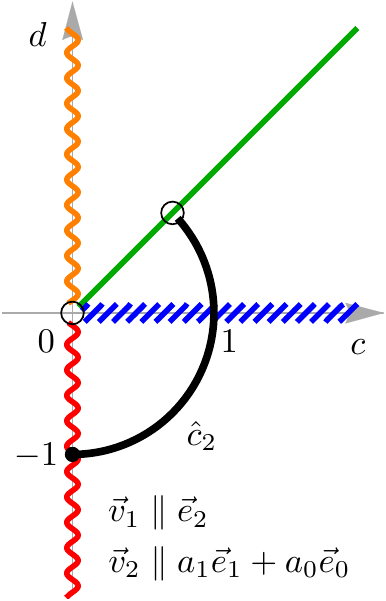}}
	\begin{subfigure}{0.35\textwidth}
		\vbox to \ht\mybox{%
			\centering
			\vfill
			\includegraphics[width=0.6\textwidth]{pathDiagDeform2.pdf}
			\vfill
		}
		\caption{Case (1b) to (1a)}
		\label{fig:pathD2}
	\end{subfigure}
	~
	\begin{subfigure}{0.35\textwidth}
		\vbox to \ht\mybox{%
			\centering
			\vfill
			\includegraphics[width=0.6\textwidth]{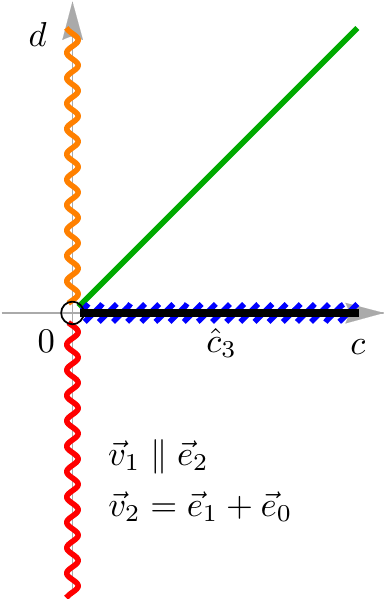}
			\vfill
		}
		\caption{Case (1b) to \textup{C\textsubscript{L}}}
		\label{fig:pathD3}
	\end{subfigure}
	\caption{$(c,d)$-paths for deformations. For the meaning of the diagram, see Figure \ref{fig:bifurcation}.}
	\label{fig:path1D}
\end{figure}

\subsection{Deformation to case (1b)}
We now show that there is a continuous deformation of all the surfaces in case (1b), and hence, all the surfaces in cases (1a) and (1b) are connected via the timelike minimal Enneper-type surface.
We first normalize the axial directions as in Section \ref{sssec:sheet2}, and let $c = \cos{\hat{c}_2}$ and $d = \sin{\hat{c}_2}$, while $a_1 = \sqrt{\cos{\hat{c}_2} - \sin{\hat{c}_2}}$ and $a_0 = \sqrt{\cos{\hat{c}_2}}$ for $\hat{c}_2 \in \left[-\tfrac{\pi}{2}, \tfrac{\pi}{4}\right]$ (see Figure \ref{fig:pathD2}).
After calculating the normal vector, we find that
	\begin{equation}\label{eqn:wDataDeform2}
		\begin{aligned}
		h_{\textup{S2}}^{\hat{c}_2}(z) &=
			\begin{cases}
				j \left(\left(\tfrac{a_0}{a_1} + 1\right)\phoro{e}^{a_1 j z} - \tfrac{a_0}{a_1}\right), &\text{if $\hat{c}_2 \neq \tfrac{\pi}{4}$},\\
				\tfrac{1}{2^{1/4}}z + j, &\text{if $\hat{c}_2 = \tfrac{\pi}{4}$},
			\end{cases}\\
		\eta_{\textup{S2}}^{\hat{c}_2}(z) &= 
			\begin{cases}
				\mathrlap{\tfrac{1}{2(a_1 + a_0)}\phoro{e}^{-a_1 j z},}\hphantom{j \left(\left(\tfrac{a_0}{a_1} + 1\right)\phoro{e}^{a_1 j z} - \tfrac{a_0}{a_1}\right),} &\text{if $\hat{c}_2 \neq \tfrac{\pi}{4}$},\\
				\tfrac{1}{2^{3/4}}, &\text{if $\hat{c}_2 \neq \tfrac{\pi}{4}$}.
			\end{cases}
		\end{aligned}
	\end{equation}

\begin{rema}
	Note that the Weierstrass data $\left\{\left(h_{\textup{S2}}^{\hat{c}_2}, \eta_{\textup{S2}}^{\hat{c}_2} \dif z\right) : \hat{c}_2 \in  \left[-\tfrac{\pi}{2}, \tfrac{\pi}{4}\right)\right\}$ describes the same set of surfaces as $\left\{\left(h_2^{c_2}, \eta_2^{c_2} \dif z\right) : c_2 \in [0, \infty)\right\}$ as in \eqref{eqn:wData2}, up to homothety and translation in the domain, the $(x, y)$-plane. Explicitly,
	\[
		h_{\textup{S2}}^{\hat{c}_2}\left(\tfrac{1}{a_1}\left(z - j \log{\left(1 + \tfrac{a_0}{a_1}\right)}\right)\right) = j\phoro{e}^{j z} - j \tfrac{a_0}{a_1} = h_2^{c_2}(z)\Big|_{c_2 = \tfrac{a_0}{a_1}}.
	\]
\end{rema}

To get the parametrization, let $F_{\textup{S2}}^{\hat{c}_2} (x,y)$ be defined from the Weierstrass data $\left(h_{\textup{S2}}^{\hat{c}_2}, \eta_{\textup{S2}}^{\hat{c}_2} \dif z\right)$ via the Weierstrass-type representation in Fact \ref{fact:Weierstrass1}, and consider
	\[
		\hat{F}_{\textup{S2}}^{\hat{c}_2} (x,y) = F_{\textup{S2}}^{\hat{c}_2} (x, y) - F_{\textup{S2}}^{\hat{c}_2} (0,0).
	\]
Then we have that
	\[
		\lim_{\hat{c}_2 \nearrow \tfrac{\pi}{4}} \hat{F}_{\textup{S2}}^{\hat{c}_2} \left(x,y - 2^{1/4}\right)  + \left(\tfrac{1}{\sqrt{2}},\, 0, \, -\tfrac{2\sqrt{2}}{3}\right) = \lim_{\hat{c}_1 \to \tfrac{\pi}{4}} \hat{F}_{\textup{P}}^{\hat{c}_1} (x, y),
	\]
implying that there is a deformation joining surfaces in case (1a) and (1b).

\subsection{Deformation to the timelike catenoid with lightlike axis}
Now we show that there exists a deformation to the timelike catenoid with lightlike axis.
Consider case (1b), where $c = {\hat{c}_3}^2$ and $d = 0$ for ${\hat{c}_3}^2 \in (0, \infty)$, and normalize the axial directions so that $\vec{v}_1 \parallel \vec{e}_2$ and $\vec{v}_2 = \vec{e}_1 + \vec{e}_0$ (see Figure \ref{fig:pathD3}).
Calculating the Weierstrass data gives
	\begin{equation}\label{eqn:wDataDeform3}
		h_{\textup{C\textsubscript{L}}}^{\hat{c}_3}(z) =
			\frac{j\left((\hat{c}_3 + 1)\phoro{e}^{j \hat{c}_3 z} - 1\right)}{(\hat{c}_3 -1)\phoro{e}^{j \hat{c}_3 z} + 1}, \quad
		\eta_{\textup{C\textsubscript{L}}}^{\hat{c}_3}(z) = 
			\tfrac{1}{4 {\hat{c}_3}^2}\phoro{e}^{-j \hat{c}_3 z} \left((\hat{c}_3 - 1)\phoro{e}^{j \hat{c}_3 z} + 1 \right)^2.
	\end{equation}
Then note that
	\[
		h_{\textup{C\textsubscript{L}}}^{\hat{c}_3}(z)\Big|_{\hat{c}_3 = 1} = 2j\phoro{e}^{j z} - j = h_{\textup{S2}}^{\hat{c}_2} (x,y)\Big|_{\hat{c}_2 = 0}, \quad
		\lim_{\hat{c}_3 \searrow 0} h_{\textup{C\textsubscript{L}}}^{\hat{c}_3}(z) = \frac{z + j}{1 - j z} = h_5(z).
	\]
Therefore, by calculating $F_{\textup{C\textsubscript{L}}}^{\hat{c}_3} (x, y)$ from $(h_{\textup{C\textsubscript{L}}}^{\hat{c}_3}(z), h_{\textup{C\textsubscript{L}}}^{\hat{c}_3}(z) \dif z)$ via Fact \ref{fact:Weierstrass1} and defining
	\[
		\hat{F}_{\textup{C\textsubscript{L}}}^{\hat{c}_3} (x,y) = F_{\textup{C\textsubscript{L}}}^{\hat{c}_3} (x, y) - F_{\textup{C\textsubscript{L}}}^{\hat{c}_3} (0,0),
	\]
we see that
	\begin{gather*}
		\hat{F}_{\textup{C\textsubscript{L}}}^{\hat{c}_3} (x,y) \Big|_{\hat{c}_3 = 1} = \hat{F}_{\textup{S2}}^{\hat{c}_2} (x,y) \Big|_{\hat{c}_2 = 0}, \\
		\lim_{\hat{c}_3 \searrow 0} \hat{F}_{\textup{C\textsubscript{L}}}^{\hat{c}_3} (x,y) = \tfrac{1}{2}\left(y - x^2 y - \tfrac{1}{3}y^3, \, -2xy, \, - y - x^2 y - \tfrac{1}{3}y^3 \right),
	\end{gather*}
implying that $\hat{F}_{\textup{C\textsubscript{L}}}^{\hat{c}_3} (x,y)$ gives a deformation between timelike minimal Bonnet-type surface with lightlike axis of first kind and timelike catenoid with lightlike axis.

\subsection{Deformation to case (1d)}
Since we have that
	\[
		h_4^{c_4}(z)\Big|_{c_4 = 0} = j \phoro{e}^{jz} = h_{\textup{S2}}^{\hat{c}_2}(z)\Big|_{\hat{c}_2 = -\frac{\pi}{2}} 
	\]
where $h_4^{c_4}$ is as in \eqref{eqn:wData4}, we define $F_{S4}^{c_4}$ using the Weierstrass data $(h_4^{c_4}, \eta_4^{c_4} \dif z)$.
Then for
	\[
		\hat{F}_{S4}^{c_4}(x,y) = F_{S4}^{c_4}(x,y) - F_{S4}^{c_4}(0,0),
	\]
we can directly check that
	\[
		\hat{F}_{S4}^{c_4}(x,y)\Big|_{c_4 = 0} = \hat{F}_{\textup{S2}}^{\hat{c}_2} (x,y)\Big|_{\hat{c}_2 = -\frac{\pi}{2}},
	\]
implying that there is a deformation joining surfaces in case (1b) and (1d).

\subsection{Deformation to the timelike minimal Bonnet-type surface with lightlike axial direction of second kind}
Finally, we show that the timelike minimal Bonnet-type surface with lightlike axial direction of second kind is also connected via a deformation to the immersed timelike catenoid with spacelike axis.
To do this, instead of re-calculating the Weierstrass data from the normal vector function, we take advantage of their respective Weierstrass data in Theorem \ref{theo:wData}, and consider
	\begin{equation}\label{eqn:wDataDeform4}
		h_{\textup{B\textsubscript{L2}}}^{\hat{c}_5} (z) = \phoro{e}^{2^{1/4}z} + j \hat{c}_5,\quad
		\eta_{\textup{B\textsubscript{L2}}}^{\hat{c}_5} (z) = \tfrac{1}{2^{5/4}}\phoro{e}^{-2^{1/4}z}
	\end{equation}
for $\hat{c}_5 \in [0,1]$.
Then it is easy to see that letting $\hat{c}_5 = 0$ gives the Weierstrass data for immersed timelike catenoid with spacelike axis, while letting $\hat{c}_5 = 1$ gives the Weierstrass data for timelike minimal Bonnet-type surface with lightlike axial direction of second kind.

Now we would like to see that the surfaces defined by $\hat{c}_5 \in (0,1)$ are also timelike minimal surfaces with planar curvature lines.
To do this, recall that the choice of the paraholomorphic $1$-form from the parameromorphic function decides the Hopf differential; therefore, a timelike minimal surface is uniquely determined by its Lorentz conformal factor up to isometries of the ambient space.
Hence, by calculating the Lorentz conformal factor from $\left(h_{\textup{B\textsubscript{L2}}}^{\hat{c}_5}, \eta_{\textup{B\textsubscript{L2}}}^{\hat{c}_5} \dif z \right)$ via \eqref{eqn:wConformal}, we find that the surfaces obtained for $\hat{c}_5 \in (0,1)$ are timelike minimal Bonnet-type surfaces with spacelike axial direction.

Using Remark \ref{rema:dpWdata} (or by directly calculating), for $F_{\textup{B\textsubscript{L2}}}^{\hat{c}_5} (x,y)$ coming from Fact \ref{fact:Weierstrass1} using the Weierstrass data $\left(h_{\textup{B\textsubscript{L2}}}^{\hat{c}_5} (z), \eta_{\textup{B\textsubscript{L2}}}^{\hat{c}_5} (z) \dif z\right)$, if we define
	\[
		\hat{F}_{\textup{B\textsubscript{L2}}}^{\hat{c}_5} (x,y) = F_{\textup{B\textsubscript{L2}}}^{\hat{c}_5} (x,y).\begin{pmatrix}
			0 & 1 & 0 \\
			-1 & 0 & 0 \\
			0 & 0 & 1
		\end{pmatrix}
			- \left( \frac{1}{\sqrt{2}},\, 0, \, 0\right),
	\]
then we have
	\[
		\hat{F}_{\textup{B\textsubscript{L2}}}^{\hat{c}_5} (x,y) \Big|_{\hat{c}_5 = 0} = \hat{F}_{\textup{P}}^{\hat{c}_1} (x,y)\Big|_{\hat{c}_1 = 0}.
	\]

Summarizing, we arrive at the following result:
\begin{theo}\label{theo:deformation}
	There exists a continuous deformation consisting exactly of all timelike minimal surfaces with planar curvature lines (see Figure \ref{fig:defoPath} and \ref{fig:deformation}).
\end{theo}

\begin{coro}[Corollary to Theorem \ref{thm:Thomsen_Bonnet_relation} and Theorem \ref{theo:deformation}]\label{cor:deformationThomsen}
There exists a continuous deformation consisting exactly of all timelike Thomsen surfaces.
\end{coro}

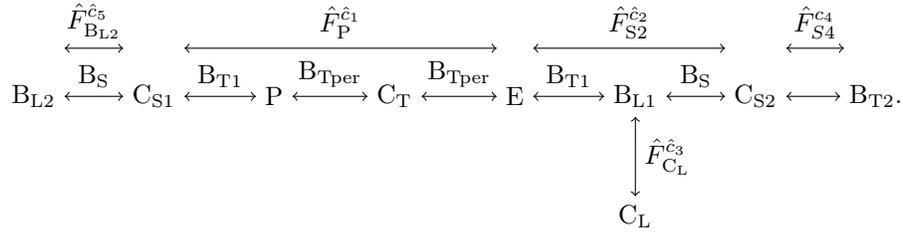
\begin{figure}
	\centering
	\begin{tikzpicture}[scale=1.6]
		\node (A) at (0,1) {\textup{B\textsubscript{L2}}};
		\node (B) at (1,1) {\textup{C\textsubscript{S1}}};
		\node (C) at (2,1) {\textup{P}};
		\node (D) at (3,1) {\textup{C\textsubscript{T}}};
		\node (E) at (4,1) {\textup{E}};
		\node (F) at (5,1) {\textup{B\textsubscript{L1}}};
		\node (G) at (6,1) {\textup{C\textsubscript{S2}}};
		\node (H) at (7,1) {\textup{B\textsubscript{T2}.}};
		\node (I) at (5,0) {\textup{C\textsubscript{L}}};
		\node (J) at (0,1.4) {$\phantom{\textup{B\textsubscript{L2}}}$};
		\node (K) at (1,1.4) {$\phantom{\textup{C\textsubscript{S1}}}$};
		\node (L) at (4,1.4) {$\phantom{\textup{E}}$};
		\node (M) at (6,1.4) {$\phantom{\textup{C\textsubscript{S2}}}$};
		\node (N) at (7,1.4) {$\phantom{\textup{B\textsubscript{T2}}}$};
		\path[<->]
			(A) edge node[above]{\textup{B\textsubscript{S}}}(B)
			(B) edge node[above]{\textup{B\textsubscript{T1}}}(C)
			(C) edge node[above]{\textup{B\textsubscript{Tper}}}(D)
			(D) edge node[above]{\textup{B\textsubscript{Tper}}}(E)
			(E) edge node[above]{\textup{B\textsubscript{T1}}}(F)
			(F) edge node[above]{\textup{B\textsubscript{S}}}(G)
			(G) edge (H)
			(J) edge node[above]{$\hat{F}_{\textup{B\textsubscript{L2}}}^{\hat{c}_5}$}(K)
			(K) edge node[above]{$\hat{F}_{\textup{P}}^{\hat{c}_1}$}(L)
			(L) edge node[above]{$\hat{F}_{\textup{S2}}^{\hat{c}_2}$}(M)
			(M) edge node[above]{$\hat{F}_{S4}^{c_4}$}(N)
			(F) edge node[right]{$\hat{F}_{\textup{C\textsubscript{L}}}^{\hat{c}_3}$} (I);
	\end{tikzpicture}
	\caption{Diagram of deformations connecting timelike minimal surfaces with planar curvature lines.}
	\label{fig:defoPath}
\end{figure}

\begin{figure}
	\centering
	\includegraphics[width=0.8\textwidth]{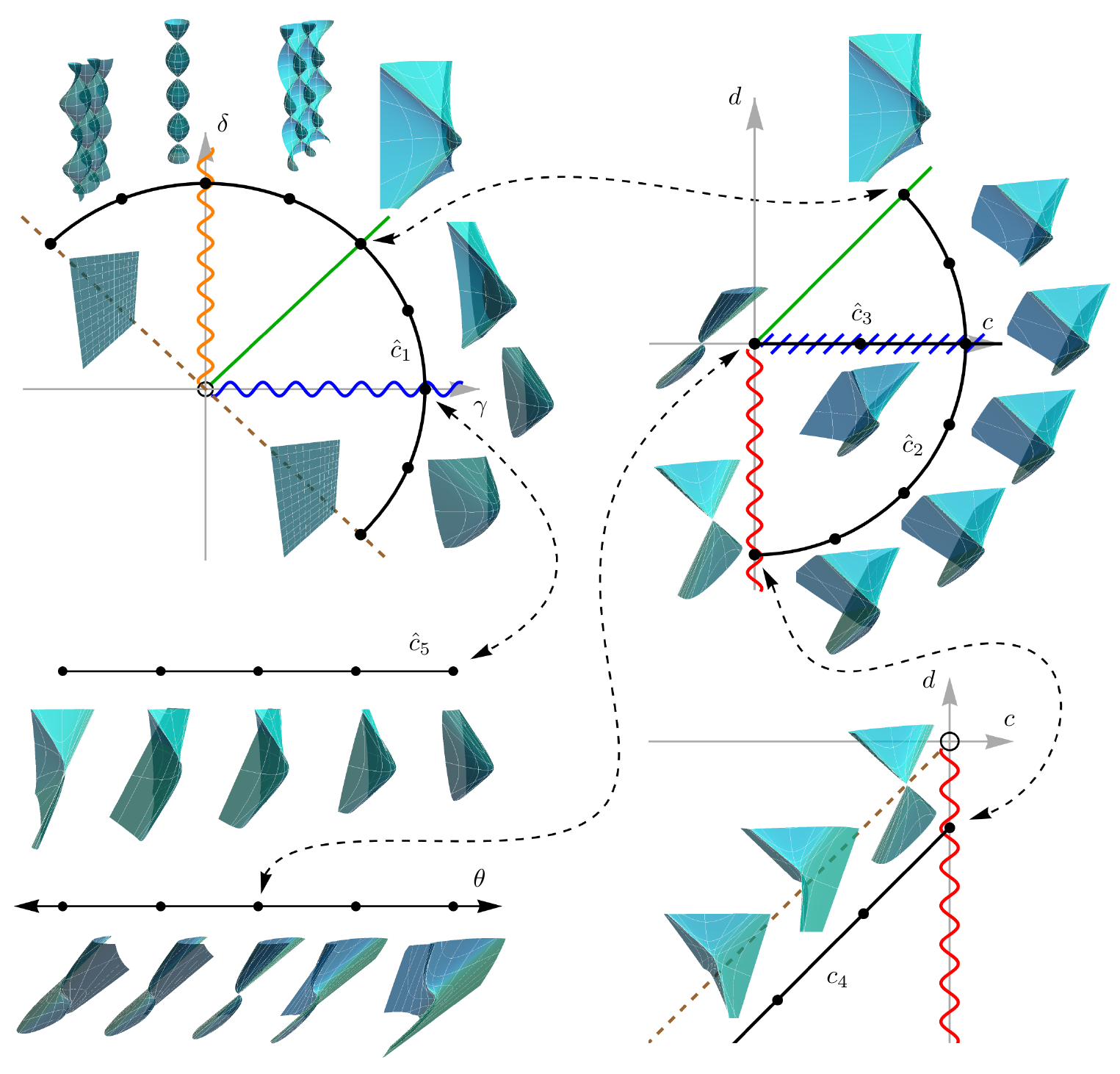}
	\caption{Continuous deformation of timelike minimal surfaces with planar curvature lines.}
	\label{fig:deformation}
\end{figure}

\section{Singularities of timelike Thomsen surfaces}\label{sect:singularities}
By Remark \ref{rema:global}, we understand that timelike minimal surfaces with planar curvature lines admit singularities, belonging to a class of surfaces called generalized timelike minimal surfaces.
However, since we have obtained the paraholomorphic $1$-form $\eta \, \dif z$ for all generalized timelike minimal surfaces with planar curvature lines in Theorem \ref{theo:wData}, we can calculate that these surfaces are actually \emph{minfaces}, using \cite[Proposition 2.7]{takahashi_tokuitenwo_2012} (see also \cite[Fact A.7]{akamine_behavior_nodate}).

We now aim to investigate the types of singularities appearing on these surfaces.
Since the types of singularities of timelike catenoids and timelike Enneper-type surfaces have been investigated in \cite[Lemma $2.12$]{kim_spacelike_2011} and \cite{takahashi_tokuitenwo_2012} (see also \cite[Example $4.5$]{akamine_behavior_nodate}), we focus on recognizing the types of singularities on timelike minimal Bonnet-type surfaces.

Let $S(F):=\{(x,y)\in\mathbb{R}^2 : \rho(x,y)=0\}=\{(x,y)\in\mathbb{R}^2 : |h(x,y)|^2=-1\}$ be the singular set.
Then using the explicit solution of the metric function in Proposition \ref{prop:solutionFG} or the explicit form of the function $h$ of the Weierstrass data in Theorem \ref{theo:wData}, we understand that the singular set becomes $1$-dimensional.
To recognize the types of singularities of timelike minimal Bonnet-type surfaces, we refer to the following results from \cite{takahashi_tokuitenwo_2012} (see also \cite[Theorem 3]{yasumoto_weierstrass-type_2018} and \cite[Fact $4.1$]{akamine_behavior_nodate}), analogous results of \cite{umehara_maximal_2006} and \cite{fujimori_singularities_2008}.

\begin{fact}\label{prop:criteria}
Let $F(x,y) : \Sigma \to \mathbb{R}^{2,1}$ be a minface with Weierstrass data $(h, \eta\,\dif z)$. Then, a point $p \in \Sigma$ is a singular point if and only if $| h(p) |^2 = -1$. Furthermore, for
\[\psi := \frac{h_z}{h^2 \eta}, \quad \Psi := \frac{h}{h_z}\psi_z,\]\ 
the image of $F$ around a singular point $p$ is locally diffeomorphic to
\begin{itemize}
	\item[$\bullet$] a cuspidal edge if and only if $\Re \psi \neq 0$ and $\Im \psi \neq 0$ at $p$, or
	\item[$\bullet$] a swallowtail if and only if $\psi \in \mathbb{R} \setminus \{0\}$ and $\Re \Psi \neq 0$ at $p$.
\end{itemize}
\end{fact}

Using the Weierstrass data $(h ,\eta \dif z)$ of timelike minimal Bonnet-type surfaces from Theorem \ref{theo:wData}, we directly calculate $\psi$ and $\Psi$.
Then using Fact \ref{prop:criteria}, we arrive at the following result.
\begin{theo}\label{theo:singularityType}
Let $F(x,y)$ be a timelike minimal Bonnet-type surface with the Weierstrass data given in Theorem \ref{theo:wData}. Then, the image of $F$ around a singular point $p = (x,y)$ is locally diffeomorphic to swallowtails (SW) only at the following points.
	\begin{center}
		{\renewcommand{\arraystretch}{1.2}
		\begin{tabular}{@{} ll @{}}
			\toprule
				Surface\ & Points of SW\\
			\midrule
				\textup{B\textsubscript{Tper}} & $\left( \cos^{-1}\left(\pm1\right), \cos^{-1}\left(\pm\tfrac{\tilde{c}_1}{\sqrt{\tilde{c}_1^2+1}}\right)\right), \left( \cos^{-1}\left(0\right), \cos^{-1}\left(\pm\frac{1}{\sqrt{\tilde{c}_1^2+1}}\right)\right)$\\
				\textup{B\textsubscript{T1}} & $\left(0, \log\left(c_2+1\right)\right)$\\ 
				\textup{B\textsubscript{L1}} & $\left(0, \log 2\right)$\\ 
				\textup{B\textsubscript{S}} & $\left(0, \log\left(c_2\pm1\right)\right)$\\
				\textup{B\textsubscript{T2}} & None\\
				\textup{B\textsubscript{L2}} & None\\		
			\bottomrule
		\end{tabular}}
	\end{center}
Moreover, the images of $F$ around singular points are locally diffeomorphic to cuspidal edges everywhere else (see Figure \ref{fig:singularities}).
\end{theo}

Combined with the result in \cite{kim_spacelike_2011, takahashi_tokuitenwo_2012, akamine_behavior_nodate}, we obtain the following corollary.
\begin{coro}\label{cor:singularities_Bonnet}
	Let $F(x,y)$ be a minface with planar curvature lines.
	If $p$ is a singular point of $F(x,y)$, then the image of $F$ around the singular point $p$ must be locally diffeomorphic to one of the following: cuspidal edge, swallowtail or conelike (or shrinking) singularity.
\end{coro}

Using the duality for singularities on timelike minimal surfaces and their conjugate surfaces, proved in \cite{kim_spacelike_2011} and \cite{takahashi_tokuitenwo_2012} (cf. \cite[Fact A.12]{akamine_behavior_nodate}), we finally obtain the following result:
\begin{coro}[Corollary to Theorem \ref{thm:Thomsen_Bonnet_relation} and Corollary \ref{cor:singularities_Bonnet}]\label{cor:singularities_Thomsen} 
	Any singular point on a timelike Thomsen surface is locally diffeomorphic to one of the following: cuspidal edge, swallowtail, cuspidal cross cap, conelike (or shrinking) singularity or fold singularity.
\end{coro}

\begin{rema}
Note that in Figure \ref{fig:lightlikeline}, the surface \textup{B\textsubscript{L2}} defined over the domain $\mathbb{C}'$ is drawn; in fact, this surface can be extended to a lightlike line (drawn as a yellow line in Figure \ref{fig:lightlikeline}) as in the cases of catenoids with spacelike and lightlike axes (\cite{fujimori_zero_2012,fujimori_zero_2015}).

To see this explicitly, first note that the surface $\circled{9}$ in Theorem \ref{theo:wData} is parametrized as 
	\[
		F(x,y)=\left(x + e^{-x}\sinh{y},- y - \tfrac{e^x}{2}\cosh{y}, - x - (e^{-x}+\tfrac{e^{x}}{2}) \sinh{y} \right).
	\]
Putting $\varrho(x) = - x - \tilde{y}$ for $\tilde{y} \in \mathbb{R}$, we note that
	\[
		\lim_{x \to -\infty} F(x,\sinh^{-1}{(e^{x}\varrho(x)})) = (- \tilde{y},  0, \tilde{y}).
	\]
\end{rema}



\vspace{15pt}
\textbf{Acknowledgements.} The authors would like to thank Professor Masaaki Umehara for encouragement in researching this topic at the conference ``The Third Japanese-Spanish Workshop on Differential Geometry'' held at Instituto de Ciencias Matem\'{a}ticas.
Also, the authors would like to thank Professor Junichi Inoguchi, Professor Wayne Rossman, and the referee for numerous valuable comments, while the first author would like to thank Professor Atsufumi Honda for various suggestions on techniques involving null curves.
Furthermore, the first and the third authors would like to express gratitude to the organizers of the aforementioned conference for their support and hospitality during the visit.
The first and the second authors were partially supported by JSPS/FWF Bilateral Joint Project I3809-N32 “Geometric shape generation"; the third author was supported by the National Institute of Technology via ``Support Program for Presentation at International Conferences'' to attend the aforementioned conference, and was also partly supported by JSPS Grant-in-Aid for Research Activity Start-up No.\ 17H07321.


\bibliographystyle{abbrvnat}

\end{document}